\documentclass[12pt,reqno]{amsart}
\usepackage{amsmath,amssymb,amsfonts}
\usepackage{eucal,enumerate}
\usepackage{graphicx,psfrag}

\newtheorem{prop}{Proposition}

\theoremstyle{remark}
\newtheorem{conj}{Conjecture}

\allowdisplaybreaks

\newcommand\ds{\displaystyle}
\newcommand\ts{\textstyle}

\newcommand\hstrut[1]{\rule{#1}{0ex}}
\newcommand\vstrut[1]{\rule{0ex}{#1}}
\newcommand\lineeq[4]{\begin{tabular}{r@{ }l}$#1$&$=#2$\\$#3$&$=#4$\end{tabular}}
\newcommand\intersects[2]{\begin{tabular}{c}$#1$\\$#2$\end{tabular}}

\newcommand\bE{\mathbf{E}}
\newcommand\bL{\mathbf{L}}
\newcommand\bM{\mathbf{M}}
\newcommand\bN{\mathbf{N}}

\newcommand\bR{\mathbf{R}}
\newcommand\bS{\mathbf{S}}
\newcommand\bff{\mathbf{f}}

\newcommand\bm{\mathbf{m}}
\newcommand\bn{\mathbf{n}}
\newcommand\bq{\mathbf{q}}
\newcommand\br{\mathbf{r}}

\newcommand\fF{{\fG_S}}
\newcommand\fG{\mathfrak{G}}

\newcommand\bbR{\mathbb{R}}
\newcommand\bbZ{\mathbb{Z}}

\newcommand\0{\hat 0}
\renewcommand{\phi}{\varphi}                 % Personal preferences.
\renewcommand{\epsilon}{\varepsilon}

\newcommand\eset{\emptyset}                  % Math abbreviations.
\renewcommand\emptyset{\varnothing}
\renewcommand\mod[1]{\ (\operatorname{mod}#1)}

\newcommand\inv{^{-1}}
\newcommand\codim{\operatorname{codim}}
\newcommand\vol{\operatorname{vol}}

\newcommand\cH{\mathcal{H}}
\newcommand\cI{\mathcal{I}}
\newcommand\cL{\mathcal{L}}
\newcommand\cS{\mathcal{S}}

\newcommand\PH{{P,\cH}}		%  Pair of polytope and arr.
\newcommand\PoH{{P^\circ,\cH}}	%  Pair of open polytope and arr.
\newcommand\lcm{\operatorname{lcm}}

\newcommand\hM{{m}}
\newcommand\hS{{s}}
\newcommand\hL{{l}}
\newcommand\bhL{\mathbf{l}}
\newcommand\hbL{\mathbf{l}}
\newcommand\hR{r}
\newcommand\hbR{\mathbf{r}}

\renewcommand\aa{{\mathrm{a}}}
\newcommand\cc{{\mathrm{c}}}
\newcommand\ma{\mathrm{ma}}
\newcommand\mc{\mathrm{mc}}
\newcommand\ml{{\mathrm{ml}}}
\newcommand\mla{{\mathrm{mla}}}
\newcommand\mlc{{\mathrm{mlc}}}
\newcommand\sss{{\mathrm{s}}}

\begin{document}
%%%%%%%%%%%%%%%%%%%%%%%%%%%%%%%%%%%%%%%%%%%%%%%%%%%%%%%%%%%%%%%%%%%%%%%%%%%%

\title{Six Little Squares and How Their Numbers Grow}

\author{Matthias Beck \\ San Francisco State University \\}
\address{Department of Mathematics, San Francisco State University, 1600 Holloway Avenue, San Francisco, CA 94132, U.S.A.}
\email{beck@math.sfsu.edu}
\author{Thomas Zaslavsky \\ Binghamton University (SUNY)}
\address{Department of Mathematical Sciences, Binghamton University, Binghamton, NY 13902-6000, U.S.A.}
\email{zaslav@math.binghamton.edu}

\begin{abstract}
 We find the numbers of $3 \times 3$ magic, semimagic, and magilatin squares, as functions either of the magic sum or of an upper bound on the entries in the square.  Our results on magic and semimagic squares differ from previous ones in that we require the entries in the square to be distinct from each other and we derive our results not by \emph{ad hoc} reasoning but from the general geometric and algebraic method of our paper ``An enumerative geometry for magic and magilatin labellings''.  Here we illustrate that method with a detailed analysis of $3\times3$ squares.
 \end{abstract}

%05A15 Exact enum., gen.\ functions.  
%05B15 Orthogonal arrays, Latin squares, Room squares [magic squares].
%05C78 Graph labelling (graceful graphs, bandwidth, etc.) [magic graphs]. 
%52B20 Lattice polytopes.  
%52C07 Lattices and convex bodies in $n$ dimensions. 
%52C35 Arrangements of points, flats, hyperplanes.


\thanks{
The authors are grateful to an anonymous referee for several helpful suggestions.
The research of the first author was partially supported by National Science Foundation grant DMS-0810105.
The research of the second author was partially supported by National Science Foundation grant DMS-0070729 and by the SGPNR.
}

\maketitle

\setcounter{tocdepth}{4}
\tableofcontents

%%%%%%%%%%%%%%%%%%%%%%%%%%%%%%%%%%%%%%%%%%%%%%%%%%%%%%%%%%%%%%%%%
\section{Introduction} \label{csintro}

``Today, the study of magic squares is not regarded as a subject of mathematics, but many earlier mathematicians in China and Japan studied it.''  These words from Shigeru's history of old Japanese mathematics \cite[p.\ 435]{Shigeru} are no longer completely true.  While the construction of magic squares remains for the most part recreational, their counting has become part of the mainstream of enumerative combinatorics, as an example of quasipolynomial counting formulas and as an application of Ehrhart's theory of lattice points in polytopes.  There are several classical \cite{MacM,Esquares,LHDE} and recent \cite{ahmeddeloerahemmecke,Bmagic} mathematical works on counting something like magic squares, but without the requirement that the entries be distinct, and often omitting the diagonals. 
In previous articles \cite{IOP,MML} we established the groundwork for an enumerative theory of magic squares with distinct entries.  Here we apply those geometrical and algebraic methods to solve the problem of counting three kinds of magical $3\times3$ squares.

Each square $x=(x_{jk})_{3\times3}$ has positive integral entries that satisfy certain line-sum equations and distinctness conditions.  In a \emph{weakly semimagic square} every row and column sum is the same (their common value is called the \emph{magic sum}); in a \emph{weakly magic square} each of the two diagonals also adds up to the magic sum.  Such squares have been studied before (see, e.g., Beck et al.\ \cite{Bmagic} and Stanley \cite{EC1}); the difference here is that we count \emph{strongly} magic or semimagic squares, where all entries of the square are distinct.  (Since strongly magic squares are closest to what are classically known as ``magic squares''---see the introduction to our general magic article \cite{MML}---we call strong squares simply ``magic'' or ``semimagic'' without qualification.)  The third type we count is a \emph{magilatin square}; this is a weakly semimagic square with the restriction that the entries be distinct within a row or column.  
The numbers of standard magic squares (with entries $1, 2, \dots, n^2$) and latin squares (in which each row or column has entries $1, 2, \dots, n$) are special evaluations of our counting functions.

We count the squares in two ways: by magic sum (an \emph{affine} count), and by an upper bound on the numbers in the square (a \emph{cubic} count).  Letting $N(t)$ denote the number of squares in terms of a parameter $t$ which is either the magic sum or a strict upper bound on the entries, we know by our previous work \cite{MML} that $N$ is a \emph{quasipolynomial}, that is, there are a positive integer $p$ and polynomials $N_0,N_1,\ldots,N_{p-1}$ so that
\[
N(t)\ =\ N_{t \mod{p}}(t) \ .
\]
The minimal such $p$ is the \emph{period} of $N$; the polynomials $N_0,N_1,\ldots,N_{p-1}$ are the \emph{constituents} of $N$, and $N_0$ is the \emph{principal constituent}.  Here we find an explicit list of constituents and also the explicit rational generating function $\bN(x) = \sum_{ t>0 } N(t) x^t$ (from which the quasipolynomial is easily extracted).  

Each magic and semimagic square also has an \emph{order type}, which is the arrangement of the cells in order of increasing value of their entries.  The order type is a linear ordering of the cells because all entries are distinct in these squares.  There are $9! = 362880$ possible linear orderings but only a handful are order types of squares.  Our approach finds the actual number of order types for each kind of square; it is the absolute value of the constant term of the principal constituent, that is, $|N_0(0)|$.  
(See Theorems 3.4 and 3.14 and Examples 3.11, 3.12, and 3.21 in our paper on magic labellings \cite{MML}.)  Obviously, this number will be the same for cubic and affine counts of the same kind of square.  (There is also an order type for magilatin squares, which is a linear ordering only within each row and column.  As it is not a simple permutation of the cells, we shall not discuss it any further.)

One of our purposes is to illustrate the technique of our general treatment \cite{MML}.  Another is to provide data for the further study of magic squares and their relatives; to this end we list the exact numbers of each type for small values of the parameter and also the numbers of symmetry types, reduced squares, and reduced symmetry types of each type (and we refer to the On-Line Encyclopedia of Integer Sequences (OEIS) \cite{OEIS} for the first 10,000 values of each counting sequence).  
A square is \emph{reduced} by subtracting the smallest entry from all entries; thus, the smallest entry in a reduced square is 0.  A square is \emph{normalized} by being put into a form that is unique in each symmetry class.  Clearly, the number of normalized squares, i.e., of equivalence classes under symmetry, is fundamental; and the number of reduced, normalized squares is more fundamental yet.  

There are other ways to find exact formulas.  Xin \cite{Xin} tackles $3\times3$ magic squares, counted by by magic sum, using MacMahon's partition calculus.  He gets a generating function that agrees with ours (thereby confirming both).  Stanley's idea of M\"obius inversion over the partition lattice \cite[Exercise 4.10]{EC1} is similar to ours in spirit, but it is less flexible and requires more computation.  Beck and van Herick \cite{BH} have counted $4\times4$ magic squares using the same basic geometrical setup as ours but with a more direct counting method.

Our paper is organized as follows.  Section \ref{tech} gives an outline of our theoretical and computational setup, as well as some comments on checks and feasibility.  In Section \ref{magic} we give a detailed analysis of our computations for counting magic $3\times 3$ squares.  Sections \ref{semimagic} and \ref{magilatin} contain the setup and the results of similar computations for $3\times 3$ semimagic and magilatin squares.  We conclude in Section \ref{questions} with some questions and conjectures.

We hope that these results, and still more the method, will interest both magic squares enthusiasts and mathematicians.

%%%%%%%%%%%%%%%%%%%%%%%%%%%%%%%%%%%%%%%%%%%%%%%%%%%%%%%%%%%%%%%%%%%%%%%%%%%%%
\section{The technique} \label{tech}

%%%%%%%%%%%%%%%%
\subsection{General method} \label{genmethod}

The means by which we solve the specific examples of $3\times3$ magic, semimagic, and magilatin squares is inside-out Ehrhart theory \cite{IOP}.  That means counting the number of $1/t$-fractional points in the interior of a convex polytope $P$ that do not lie in any of a certain set $\cH$ of hyperplanes.  The number of such points is a quasipolynomial function $E^\circ_\PoH(t)$, the \emph{open Ehrhart quasipolynomial} of the open inside-out polytope $(\PoH)$.  The exact polytope and hyperplanes depend on which of the six problems it is, but we can describe the general picture.  
 First, there are the equations of magic; they determine a subspace $s$ of all $3\times3$ real matrices which we like to call the \emph{magic subspace}---though mostly we work in a smaller overall space $\bbR^d$ that results from various reductions.  Then there is the polytope $P$ of constraints, which is the intersection with $s$ of either a hypercube $[0,1]^{3^2}$ or a standard simplex $\{ x\in\bbR^{3^2} : x \geq 0,\ {\sum x_{jk}=1} \}$: the former when we impose an upper bound on the magic square entries and the latter when we predetermine the magic sum.  
The parameter $t$ is the strict upper bound in the former case (which we call \emph{cubical} due to the shape of $P$), the magic sum in the latter (which we call \emph{affine} as $P$ lies in a proper affine subspace).  Finally there are the strong magical exclusions, the hyperplanes that must be avoided in order to ensure the entries are distinct---or in the magilatin examples, as distinct as they ought to be.  These all have the form $x_{ij}=x_{kl}$. 
 The combination of $P$ and the excluded hyperplanes forms the \emph{vertices} of $(\PH)$, which are all the points of intersection of facets of $P$ and hyperplanes in $\cH$ that lie in or on the boundary of $P$.  Thus, we count as a vertex every vertex of $P$ itself, each point that is the intersection of some facets and some hyperplanes in $\cH$, and any point that is the intersection of some hyperplanes and belongs to $P$, but not intersection points that are outside $P$.  (Points of each kind do occur in our examples.)  The \emph{denominator} of $(\PH)$ is the least common denominator of all the coordinates of all the vertices of $(\PH)$.  The period of $E^\circ_\PoH(t)$ divides the denominator; this gives us a known bound on it.

This geometry might best be explained with an example.  Let us consider magic $3 \times 3$ squares, 
\[
  \left[ \begin{array}{ccc}
  x_{ 11 } & x_{ 12 } & x_{ 13 } \\
  x_{ 21 } & x_{ 22 } & x_{ 23 } \\
  x_{ 31 } & x_{ 32 } & x_{ 33 }
  \end{array} \right] \in \bbZ_{ >0 }^{3^2} .
\]
The magic subspace is
\[
  s = \left\{ 
  \left[ \begin{array}{ccc}
  x_{ 11 } & x_{ 12 } & x_{ 13 } \\
  x_{ 21 } & x_{ 22 } & x_{ 23 } \\
  x_{ 31 } & x_{ 32 } & x_{ 33 }
  \end{array} \right] \in \bbR^{3^2} : \
  \begin{array}{l}
  x_{ 11 } + x_{ 12 } + x_{ 13 } = x_{ 21 } + x_{ 22 } + x_{ 23 } \\
 = x_{ 31 } + x_{ 32 } + x_{ 33 }  = x_{ 11 } + x_{ 21 } + x_{ 31 } \\
 = x_{ 12 } + x_{ 22 } + x_{ 32 } = x_{ 13 } + x_{ 23 } + x_{ 33 } \\
 = x_{ 11 } + x_{ 22 } + x_{ 33 } = x_{ 13 } + x_{ 22 } + x_{ 31 } 
  \end{array}
  \right\} .
\]
The hyperplane arrangement $\cH$ that captures the distinctness of the entries is
\[
  \cH = \left\{ (x_{ 11 } = x_{ 12 }) \cap s ,\ (x_{ 11 } = x_{ 21 }) \cap s ,\ \dots,\ (x_{ 32 } = x_{ 33 }) \cap s \right\} \ .
\]
Finally, there are two polytopes associated to magic $3 \times 3$ squares, depending on whether we count them by an upper bound on the entries:
\[
  P_\cc = s \cap [0,1]^{3^2},
\]
or by magic sum:
\[
  P_\aa = s \cap \left\{ 
  \left[ \begin{array}{ccc}
  x_{ 11 } & x_{ 12 } & x_{ 13 } \\
  x_{ 21 } & x_{ 22 } & x_{ 23 } \\
  x_{ 31 } & x_{ 32 } & x_{ 33 }
  \end{array} \right] \in \bbR_{ \ge 0 }^{3^2} : \  
  x_{ 11 } + x_{ 12 } + x_{ 13 } = 1 \right\}  .
\]
 Our cubical counting function computes the number of magic squares all of whose entries satisfy $0 < x_{ ij } < t$, in terms of an integral parameter $t$. These squares are the lattice points in
\[
  \big( P_\cc^\circ \setminus \bigcup\cH \big) \cap \big( \tfrac 1 t \bbZ \big)^{3^2} .
\]
Our second, affine, counting function computes the number of magic squares with positive entries and magic sum $t$. These squares are the lattice points in
\[
  \big( P_\aa^\circ \setminus \bigcup\cH \big) \cap \big( \tfrac 1 t \bbZ \big)^{3^2} .
\]

 In general, the number of squares we want to count, $N(t)$, is the Ehrhart quasipolynomial $E^\circ_\PoH(t)$ of an open inside-out polytope $(\PoH)$.  We obtain the necessary Ehrhart quasipolynomials by means of the computer program {\tt LattE} \cite{LattE}.  It computes the closed Ehrhart generating function
\[
  \bE_P(x) := 1 + \sum_{t=1}^\infty E_P(t) x^t \qquad \text{ of the values } \qquad E_P(t) := \# \big( P \cap \left( \tfrac 1 t \bbZ \right)^d \big) .
\]
Counting only interior points gives the open Ehrhart quasipolynomial $E_{P^\circ}(t)$ and its generating function 
\[
  \bE_{P^\circ}(x) := \sum_{t=1}^\infty E_{P^\circ}(t) x^t .
\]
Since we want the open inside-out Ehrhart generating function $\bE^\circ_\PoH(x) = \sum_{t=1}^\infty E^\circ_\PoH(t) x^t$, we need several transformations.  One is \emph{Ehrhart reciprocity} \cite{Ehr}, which is the following identity of rational generating functions:
\begin{equation}\label{E:ehrrecipgf}
\bE_{P^\circ}(t) = (-1)^{1+\dim P}\, \bE_P(x\inv) \ .
\end{equation}
The inside-out version \cite[Equation (4.6)]{IOP} is 
\begin{equation}\label{E:recipgf}
\bE^\circ_\PoH(x) = (-1)^{1+\dim P}\, \bE_{\PH}(x\inv) \ .
\end{equation}
We need to express the inside-out generating functions in terms of ordinary Ehrhart generating functions.  To do that we take the intersection poset 
\[
\cL(\PoH) := \big\{ P^\circ \cap {\ts\bigcap} \cS : \cS \subseteq \cH \big\} \setminus \big\{ \eset \big\},
\]
which is ordered by reverse inclusion.  
Note that $\cL(\PoH)$ and $\cL(\PH)$, defined similarly but with $P$ instead of $P^\circ$, are isomorphic  posets because $\cH$ is transverse to $P$; specifically, $\cL(\PH) = \{ \bar u :  u \in \cL(\PoH) \}$, where $\bar u$ is the (topological) closure of $u$.
Now we have the M\"obius inversion formulas \cite[Equations (4.7) and (4.8)]{IOP} 
\begin{equation}\label{E:oehrhypgf}
\bE^\circ_\PoH(x) = \sum_{u\in \cL(\PoH)} \mu(\0,u)\, \bE_{u}(x)
\end{equation}
and 
\begin{equation} \label{E:ehrhypgf}
\bE_\PH(x) = \sum_{u\in \cL(\PH)} |\mu(\0,u)|\, \bE_{u}(x) 
\end{equation}
(since $\cH$ is transverse to $P$; see our general paper \cite{MML}).  Here $\mu$ is the M\"obius function of $\cL(\PoH)$ \cite{EC1}.  

Thus we begin by getting all the cross-sectional generating functions $\bE_{u}(x)$ from {\tt LattE}.  Then we either sum them by \eqref{E:ehrhypgf} and apply inside-out reciprocity \eqref{E:recipgf}, or apply ordinary reciprocity \eqref{E:ehrrecipgf} first and then sum by \eqref{E:oehrhypgf}.  (We did whichever of these seemed more convenient.)  In the semimagic and magilatin counts we need a third step because the generating functions we computed pertain to a reduced problem; those of the original problem are obtained through multiplication by another generating function.

Once we have the generating function we extract the quasipolynomial, essentially by the binomial series.  If an Ehrhart quasipolynomial $q$ of a rational convex polytope has period $p$ and degree $d$, then its generating function $\bq$ can be written as a rational function of the form
\begin{equation}\label{E:quasiratfct}
\bq (x) := \sum_{ t \ge 1 } q(t) \, x^t = 
\frac{ a_{p(d+1)} \, x^{p(d+1)} + a_{p(d+1)-1} \, x^{p(d+1)-1} + \dots + a_1 \, x }{ \left( 1 - x^p \right)^{d+1} }
\end{equation}
for some nonnegative integers $a_1, a_2, \dots, a_{p(d+1)}$.  Grouping the terms in the numerator of \eqref{E:quasiratfct} according to the residue class of the degree modulo $p$ and expanding the denominator, we get
\begin{align*}
\bq (x)
&= \frac{ \sum_{r=1}^{p} \sum_{ j=0 }^{ d } a_{ pj+r } \, x^{ pj+r } }{ \left( 1 - x^p \right)^{ d+1 } } 
= \sum_{r=1}^{p} \sum_{ k \ge 0 } \left[ \sum_{ j=0 }^{ d } a_{ pj+r } \binom{ d+k-j }{ d } \right] x^{ pk+r } \ .
\end{align*}
Hence the $r$th constituent of the quasipolynomial $q$ is
\[
  q_r(t) = \sum_{ j=0 }^{ d } a_{ pj+r } \binom{ d+\frac{t-r}{p}-j }{ d } \quad \text{ for } r = 1,\ldots,p .
\] 

%%%%%%%%%%%%%%%%
\subsection{How we apply the method} \label{specificmethod}

The initial step is always to reduce the size of the problem by applying symmetry.  Each problem has a \emph{normal} form under symmetry, which is a strong square.  The number of all magic or semimagic squares is a constant multiple of the number of symmetry types, because every such square has the same symmetry group.  For magilatin squares, there are several symmetry types with symmetry groups of different sizes, so each type must be counted separately.

Semimagic and magilatin squares also have an interesting \emph{reduced} form, in which the values are shifted by a constant so that the smallest cell contains 0; and a reduced normal form; the latter two are not strong but are aids to computation.  Reduced squares are counted either by magic sum (the ``affine'' counting rule) or by the largest cell value (the ``cubic'' count).  
All reduced normal semimagic squares correspond to the same number of unreduced squares, while the different symmetry types of magilatin square give reduced normal squares whose corresponding number of unreduced squares depends on the symmetry type.

The total number of squares, $N(t)$, and the number of reduced squares, $R(t)$, are connected by a convolution identity $N(t) = \sum_s f(t-s)R(s)$, where $f$ is a periodic constant (by which we mean a quasipolynomial of degree 0; we say \emph{constant term} for the degree-0 term of a quasipolynomial, even though the ``constant term'' may vary periodically) or a linear polynomial.  Writing for the generating functions $\bN(x) := \sum_{t>0} N(t) x^t$ and similarly $\bR(x)$, and $\bff(x) := \sum_{t\geq0} f(t) x^t$, we have $\bN(x) = \bff(x)\bR(x).$  
It follows from the form of the denominator in Equation \eqref{E:quasiratfct} that the period of $N$ divides the product of the periods of $f$ and $R$.  The reduced number $R(t)$ is, in the semimagic case, a constant multiple of the number $n(t)$ of reduced, normal semimagic squares; in the magilatin case it is a sum of different multiples, depending on a symmetry group, of the number of reduced, normal magilatin squares of each different type $T$.  Each $n(t)$ is the open Ehrhart quasipolynomial $E^\circ_{Q^\circ,\cI}(t)$ of an inside-out polytope $(Q,\cI)$ which is smaller than the original polytope $P$.  

We compute $\bn(x) := \sum_{t>0} n(t) x^t$ from the Ehrhart generating functions $\bE_{u}(x)$ of the nonvoid sections $u$ of $Q^\circ$ by flats of $\cI$ through the following procedure:
\begin{enumerate}
\item  We calculate the flats and sections by hand.
\item  We feed each $u$ into the computer program {\tt LattE} \cite{LattE}, which returns the closed generating function $\bE_{\bar u}(x)$, whose constant term equals $1$ because $u$ is nonvoid and convex.
\item  With semimagic squares, by Equations \eqref{E:ehrrecipgf}--\eqref{E:ehrhypgf} we have the M\"obius-inversion formulas  
\begin{equation}
\label{E:gfmobius}
\begin{aligned}
\bn(x) = \bE^\circ_{Q^\circ,\cI}(x) 
&= \sum_{u\in\cL} \mu_\cL(\0,u) \bE_{u}(x) \\
&= (-1)^{1+\dim Q} \sum_{u\in\cL} |\mu_\cL(\0,u)|\, \bE_{\bar u}(x\inv) .
\end{aligned}
\end{equation}
where $\cL := \cL(Q^\circ,\cI)$, the intersection poset of $(Q^\circ,\cI)$.
\end{enumerate}
The procedure for magilatin squares is similar but taking account of the several types.

%%%%%%%%%%%%%%%%
\subsection{Checks}\label{checks}

We check our results in a variety of ways.

The degree is the dimension of the polytope, or the number of independent variables in the magic-sum equations.

The leading coefficient is the volume of the polytope.  (The volume is normalized so that a fundamental domain in the affine space spanned by the polytope has unit volume.)  This check is also not difficult.  The volume is easy to find by hand in the magic examples.  In affine semimagic the polytope is the Birkhoff polytope $B_3$, whose volume is well known (Section \ref{sa3}).  The cubical semimagic volume is not well known but it was easy to find (Section \ref{weaksc3}).  The magilatin polytopes are the same as the semimagic ones.

The firmest verification is to compare the results of the generating function approach with those of direct enumeration.  If we count the squares individually for $t \leq t_1$ where $t_1 \geq pd$, only the correct quasipolynomial can agree with the counts (given that we know the degree $d$ and period $p$ from the geometry). 
Though $t_1 = pd$ is too large to reach in some of the examples, still we gain considerable confidence if even a smaller value of $t_1$ yields numbers that agree with those derived from the quasipolynomial or generating function.  We performed this check in each case.

%%%%%%%%%%%%%%%%
\subsection{Feasibility}\label{feasibility}

Based on our solutions of the six $3\times3$ examples we believe our counting method is practical.  The calculations are simple and readily verified.  Linear algebra tells us the degree, geometry tells us the period; we obtain the generating function using the Ehrhart package {\tt LattE} \cite{LattE} and then apply reciprocity (Equation \eqref{E:ehrrecipgf} or \eqref{E:recipgf}) and M\"obius inversion (Equation \eqref{E:oehrhypgf} or \eqref{E:ehrhypgf}), and extract the constituents, all with {\tt Maple}.  The programming is not too difficult.

In the magic square problems we found the denominator by calculating the vertices of the inside-out polytope.  Then we took two different routes.  In one we applied {\tt LattE} and Equation \eqref{E:oehrhypgf}.  In the other we calculated $N(t)$ for small values of $t$ by generating all magic squares, taking enough values of $t$ that we could fit the quasipolynomial constituents to the data.  This was easy to program accurately and quick to compute, and it gave the same answer.  The programs can be found at our ``Six Little Squares'' Web site \cite{Maplefiles}.

In principle the semimagic and magilatin problems can be solved in the same two ways.  The geometrical method with M\"obius inversion gave complete answers in a few minutes of computer time after a simple hand analysis of the geometry (see Section \ref{semimagic}).  
Direct enumeration on the computer proved unwieldy (at best), especially in the affine case, where the period is largest.  A straightforward computer count of semimagic squares by magic sum (performed in {\tt Maple}---admittedly not the language of choice---on a personal computer) seemed destined to take a million years.  Switching to a count of reduced normal squares, the calculation threatened to take only a thousand years.  These programs are at our ``Six Little Squares'' Web site \cite{Maplefiles}, as is a complicated ``supernormalized formula'' that greatly speeds up affine semimagic counting (see Section \ref{sa3alt}).

%%%%%%%%%%%%%%%%
\subsection{Notation}\label{notation}

We use a lot of notation.  To keep track of it we try to be reasonably systematic.
\begin{itemize}
\item[] 
\begin{itemize}
\item[$M, m$] refer to magic squares (Section \ref{magic}).
\item[$S, s$] and subscript $\sss$ refer to semimagic squares (Section \ref{semimagic}).
\item[$L, l$] and subscript $\ml$ refer to magilatin squares (Section \ref{magilatin}).
\item[$R, r$] refer to reduced squares (the minimum entry is $0$), while $M, S, L,$ et al.\ refer to ordinary squares (all positive entries).
\item[$\cc$] refers to ``cubic'' counting, by an upper bound on the entries.
\item[$\aa$] refers to ``affine'' counting, by a specified magic sum.
\item[$X$] (capital) refers to all squares of that type.
\item[$x$] (minuscule) refers to symmetry types of squares, or equivalently normalized squares.
\end{itemize}
\end{itemize}

%%%%%%%%%%%%%%%%%%%%%%%%%%%%%%%%%%%%%%%%%%%%%%%%%%%%%%%%%%%%%%%%%
\section{Magic squares of order 3} \label{magic}

The standard form of a magic square of order 3 is well known; it is
\begin{equation} 
\label{E:magicnormalized}
\begin{tabular}{|c|c|c|}
\hline
\hstrut{12ex}	& \hstrut{12ex}	& \hstrut{12ex}	\\
\raisebox{-5ex}{\vstrut{8ex}}
$\alpha+\gamma$		&$-\alpha-\beta+\gamma$	&$\beta+\gamma$	\\
\hline &&\\
\raisebox{-5ex}{\vstrut{8ex}}
$-\alpha+\beta+\gamma$		&$\gamma$	&$\alpha-\beta+\gamma$	\\
\hline &&\\
\raisebox{-5ex}{\vstrut{8ex}}
$-\beta+\gamma$	&$\alpha+\beta+\gamma$		&$-\alpha+\gamma$	\\
\hline
\end{tabular}
\end{equation}
where the magic sum is $s=3\gamma$.  Taking account of the 8-fold symmetry, under which we may assume the largest corner value is $\alpha+\gamma$ and the next largest is $\beta+\gamma$, and the distinctness of the values, we have $\alpha > \beta > 0$ and $\alpha \neq 2\beta$.  One must also have $\gamma > \alpha + \beta$ to ensure positivity.

In this pair of examples, the dimension of the problem is small enough that there is no advantage in working with the reduced normal form (where $\gamma=0$).

%%%%%%%%%%%%%%%%
\subsection{Magic squares:  Cubical count (by upper bound)} \label{m3c}

Here we count by a strict upper bound $t$ on the permitted values; since the largest entry is $\alpha+\beta+\gamma$, the bound is $\alpha+\beta+\gamma<t$.  The number of squares with upper bound $t$ is $M_\cc(t)$.  We think of each magic square as a $t\inv$-lattice point in $P_\cc^\circ \setminus \bigcup\cH_\cc$, the (relative) interior of the inside-out polytope
\begin{align*}
P_\cc &:= \{(x,y,z) : 0 \leq y \leq x,\ x+y \leq z,\ x+y+z \leq 1 \} , \\
\cH_\cc &:= \{ h \} \text{ where } h: x=2y ,
\end{align*}
but multiplied by $t$ to make the entries integers.  Here we use normalized coordinates $x=\alpha/t$, $y=\beta/t$, and $z=\gamma/t$.  The semilattice of flats is $\cL(P_\cc^\circ,\cH_\cc) = \{ P_\cc^\circ, h \cap P_\cc^\circ \}$ with $P_\cc^\circ < h \cap P_\cc^\circ.$  The vertices are
$$
\ts
O=(0,0,0), \quad C=(0,0,1), \quad D=(\frac12,0,\frac12), \quad 
E=(\frac13,\frac16,\frac12), \quad F=(\frac14,\frac14,\frac12),
$$
of which $O, C, D, F$ are the vertices of $P_\cc$ and $O,C,E$ are those of $h \cap P_\cc$.  (Both these polytopes are simplices.)  From Equation \eqref{E:oehrhypgf}, 
\begin{align*}
\bM_\cc(x) &= 8 \bE^\circ_{P_\cc^\circ,\cH_\cc}(x) 
= 8 \big[ \bE_{P_\cc^\circ}(x) - \bE_{h \cap P_\cc^\circ}(x) \big] 
\intertext{which we evaluate by {\tt LattE} and Ehrhart reciprocity, Equation \eqref{E:ehrrecipgf}:}
&= 8 \left[ \frac{ x^8 }{ (1-x)^2 (1-x^2) (1-x^4) } - \frac{ x^8 }{ (1-x)^2 (1-x^6) } \right] \\[6pt]
&= \frac{ 8 x^{10} (2 x^2 + 1) }{ (1-x)^2 (1-x^4) (1-x^6) } \\[6pt]
&= \frac{ 8 x^{10} (2x^2+1) (x^4-x^2+1) (x^{11}+x^{10}+\cdots+x+1)^2 (x^{10}+x^8+\cdots+x^2+1) }{ (1-x^{12})^4 } \ .
\end{align*}
 From this generating function we extract the quasipolynomial
\begin{equation}\label{E:mc-constituents}
M_\cc(t) = \begin{cases}
\frac{t^3-16t^2+76t-96}{6} = 
\frac{(t-2)(t-6)(t-8)}{6}\,, &\text{if } t \equiv 0,2,6,8 \mod{12} ; \\[6pt]
\frac{t^3-16t^2+73t-58}{6} = 
\frac{(t-1)(t^2-15t+58)}{6}\,, &\text{if } t \equiv 1 \mod{12} ; \\[6pt]
\frac{t^3-16t^2+73t-102}{6} = 
\frac{(t-3)(t^2-13t+34)}{6}\,, &\text{if } t \equiv 3,11 \mod{12} ; \\[6pt]
\frac{t^3-16t^2+76t-112}{6} = 
\frac{(t-4)(t^2-12t+28)}{6}\,, &\text{if } t \equiv 4,10 \mod{12} ; \\[6pt]
\frac{t^3-16t^2+73t-90}{6} = 
\frac{(t-2)(t-5)(t-9)}{6}\,, &\text{if } t \equiv 5,9 \mod{12} ; \\[6pt]
\frac{t^3-16t^2+73t-70}{6} = 
\frac{(t-7)(t^2-9t+10)}{6}\,, &\text{if } t \equiv 7 \mod{12} ;
\end{cases}
\end{equation}
and the first few nonzero values for $t>0$:
$$
\begin{tabular}{r||c|c|c|c|c|c|c|c|c|c|c|c|c|c|c|c|c}
$t$     
&10&11&12&13&14&15&16&17&18&19&20&21&22&23&24\\
\hline \hline
$M_\cc(t)$&8&16&40&64&96&128&184&240&320&400&504&608&744&880&1056\\
\hline
$\hM_\cc(t)$&1&2&5&8&12&16&23&30&40&50&63&76&93&110&132\\
\end{tabular}
$$
The last row is the number of symmetry classes, or normal squares, i.e., $M_\cc(t)/8$.  The rows are sequences A108576 and A108577 in the OEIS \cite{OEIS}.  

The symmetry of the constituents about residue $1$ is curious.

The principal constituent is 
\[
\frac{t^3-16t^2+76t-96}{6} = \frac{(t-2)(t-6)(t-8)}{6}\, .
\]
Its unsigned constant term, 16, is the number of linear orderings of the cells that are induced by magic squares.  Thus, up to the symmetries of a magic square, there are just two order types, even allowing arbitrarily large cell values.  (The order types are illustrated in Example 3.11 of our general magic article, \cite{MML}.)

We confirmed these results by direct enumeration, counting the strongly magic squares for $t \leq 60$ \cite{Maplefiles}.

Compare this quasipolynomial to the weak quasipolynomial: 
$$
\begin{cases}
\frac{t^3 - 3t^2 + 5t - 3}{6} = 
\frac{(t - 1)(t^2 - 2t + 3)}{6}\,, &\text{if $t$ is odd}; \\[6pt]
\frac{t^3 - 3t^2 + 8t - 6}{6} = 
\frac{(t - 1)(t^2 - 2t + 6)}{6}\,, &\text{if $t$ is even};
\end{cases}
$$
with generating function
$$
\frac{ (x^2 + 2x - 1) (2x^3 - x^2 + 1) }{ (1-x)^2 (1-x^2)^2 }\ .
$$

%%%%%%
\subsubsection{Reduced magic squares}\label{redmagicc}

A more fundamental count than the number of magic squares with an upper bound is the number of reduced squares.  Let $R_\mc(t)$ be the number of $3\times3$ reduced magic squares with maximum cell value $t$, and $r_\mc(t)$ the number of reduced symmetry types, or equivalently of normalized reduced squares with maximum $t$.  Then we have the formulas
\[
M_\cc(t) = \sum_{k=0}^{t-1} (t-1-k) R_\mc(k) \quad \text{ and } \quad m_\cc(t) = \sum_{k=0}^{t-1} (t-1-k) r_\mc(k),
\]
since every reduced square with maximum $k$ gives $t-1-k$ unreduced squares with maximum $< t$ (and positive entries) by adding $l$ to each entry where $1 \leq l \leq t-1-k$.  In terms of generating functions,
\begin{equation}\label{E:redmagic}
\bM_\cc(x) = \frac{x^2}{(1-x)^2} \bR_\mc(x) \quad \text{ and } \quad \bm_\cc(x) = \frac{x^2}{(1-x)^2} \br_\mc(x).
\end{equation}
We deduce the generating functions 
$$
\br_\mc(x) = \frac{ x^8 (2 x^2 + 1) }{ (1-x^4) (1-x^6) }
$$ 
and $\bR_\mc(x) = 8\br_\mc(x)$, and from the latter the quasipolynomial
\begin{equation}\label{E:mrc-constituents}
R_\mc(t) = \begin{cases}
\left.\begin{aligned}
&2t-16,  &&\text{if } t \equiv 0  \\
&2t-4, &&\text{if } t \equiv 2, 10  \\
&2t-8, &&\text{if } t \equiv 4, 8  \\
&2t-12, &&\text{if } t \equiv 6   
\end{aligned}\ \right\} &\text{mod}\,{12}; \\
\ 0,  \qquad \quad\ \text{if $t$ is odd} 
\end{cases}
\end{equation}
($1/8$-th of these for $r_\mc(t)$) as well as the first few nonzero values:
$$
\begin{tabular}{r||c|c|c|c|c|c|c|c|c|c|c|c|c|c|c|c|c|c|}
$t$	&8&10&12&14&16&18&20&22&24&26&28&30&32&34&36&38&40&42\\
\hline \hline
$R_\mc(t)$&8&16&8&24&24&24&32&40&32&48&48&48&56&64&56&72&72&72\\
\hline
$r_\mc(t)$&1&2&1&3&3&3&4&5&4&6&6&6&7&8&7&9&9&9\\
\end{tabular}
$$
The sequences are A174256 and A174257 in the OEIS \cite{OEIS}.  

The principal constituent is $2t-16,$ whose constant term in absolute value, 16, is the number of linear orderings of the cells that are induced by magic squares---necessarily, the same number as with $M_\cc(t)$.

We confirmed the constituents by testing them against the coefficients of the generating function for several periods.

Our way of reasoning, from all squares to reduced squares, is backward; logically, one should count reduced squares and then deduce the ordinary magic square numbers from them via Equation \eqref{E:redmagic}.  Counting magic squares is not hard enough to require that approach, but in treating semimagic and magilatin squares we follow the logical progression since then reduced squares are much easier to handle.

%%%%%%%%%%%%%%%%
\subsection{Magic squares:  Affine count (by magic sum)} \label{m3a}

The number of magic squares with magic sum $t=3\gamma$ is $M_\aa(t)$.  We take the normalized coordinates $x=\alpha/t$ and $y=\beta/t$.  The inside-out polytope is 
\begin{align*}
P_\aa&:\ 0 \leq y \leq x,\ x+y \leq \tfrac13,	\\
\cH_\aa&:\ \{ h \} \text{ where } h: x=2y .
\end{align*}
The semilattice of flats is $\cL(P_\aa^\circ,\cH_\aa) = \{ P_\aa^\circ, h \cap P_\aa^\circ \}$ with $P_\aa^\circ < h \cap P_\aa^\circ.$  The vertices are 
$$
\ts
O=(0,0), \quad D=(\frac13,0), \quad E=(\frac29,\frac19), \quad F=(\frac16,\frac16),
$$
of which $O, D, F$ are the vertices of $P_\aa$ and $O, E$ are the vertices of $h\cap P_\aa$.  From Equations \eqref{E:ehrrecipgf}--\eqref{E:oehrhypgf},
\begin{align*}
\bM_\aa(x) &= 8 \bE^\circ_{P_\aa^\circ,\cH_\aa}(x) 
= 8 \big[ \bE_{P_\aa^\circ}(x) - \bE_{h \cap P_\aa^\circ}(x) \big]  \\
&= 8 \left[ \frac{ x^{12} }{ (1-x^3)^2 (1-x^6) } - \frac{ x^{12} }{ (1-x^3) (1-x^9) }  \right] \\[6pt]
&= \frac{ 8 x^{15} (2x^3+1) }{ (1-x^3) (1-x^6) (1-x^9) } \\[6pt]
&= \frac{ 8 x^{15} (2x^3+1) (x^9+1) (x^{12}+x^{6}+1) (x^{15}+x^{12}+\cdots+x^3+1) 
}{ (1-x^{18})^{3} } \ .
\end{align*}
 From this generating function we extract the quasipolynomial
\begin{equation}\label{E:ma-constituents}
M_\aa(t) = \begin{cases}
\frac{2t^2-32t+144}{9} = 
\frac{2}{9}(t^2-16t+72),  &\text{if } t \equiv 0 \mod{18} ; \\[6pt]
\frac{2t^2-32t+78}{9} = 
\frac{2}{9}(t-3)(t-13), &\text{if } t \equiv 3 \mod{18} ; \\[6pt]
\frac{2t^2-32t+120}{9} = 
\frac{2}{9}(t-6)(t-10), &\text{if } t \equiv 6 \mod{18} ; \\[6pt]
\frac{2t^2-32t+126}{9} = 
\frac{2}{9}(t-7)(t-9), &\text{if } t \equiv 9 \mod{18} ; \\[6pt]
\frac{2t^2-32t+96}{9} = 
\frac{2}{9}(t-4)(t-12), &\text{if } t \equiv 12 \mod{18} ; \\[6pt]
\frac{2t^2-32t+102}{9} = 
\frac{2}{9}(t^2-16t+51),  &\text{if } t \equiv 15 \mod{18} ; \\[6pt]
0,  &\text{if } t \not\equiv 0 \mod{3} ;
\end{cases}
\end{equation}
and the first few nonzero values for $t>0$:
$$
\begin{tabular}{r||c|c|c|c|c|c|c|c|c|c|c|c|c|c|}
$t$	&15&18&21&24&27&30&33&36&39&42&45&48&51&54\\
\hline \hline
$M_\aa(t)$&8&24&32&56&80&104&136&176&208&256&304&352&408&472\\
\hline
$\hM_\aa(t)$&1&3&4&7&10&13&17&22&26&32&38&44&51&59\\
\end{tabular}
$$
The last row is the number of symmetry classes, or normalized squares, which is $M_\aa(t)/8$.  The two sequences are A108578 and A108579 in the OEIS \cite{OEIS}.  

The principal constituent is 
\[
\frac{2t^2-32t+144}{9} = \frac{2}{9}(t^2-16t+72) ,
\]
whose constant term, 16, is the number of linear orderings of the cells that are induced by magic squares---the same number as with $M_\cc(t)$.

We verified our results by direct enumeration, counting the strong magic squares for $t \leq 72$ \cite{Maplefiles}.

Compare the magic-square quasipolynomial to the weak quasipolynomial:
\[
\begin{cases}
\frac{2t^2-6t+9}{9}\,, &\text{if } t \equiv 0 \mod{3} ; \\
0,  &\text{if } t \not\equiv 0 \mod{3} ;
\end{cases}
\]
due to MacMahon \cite[Vol.\ II, par.\ 409, p.\ 163]{MacM}, with generating function
\[
\frac{ 5x^6-2x^3+1 }{ (1-x^3)^3 } \ .
\]

%%%%%%
\subsubsection{Reduced magic squares}\label{redmagica}

Let $R_\ma(t)$ be the number of $3\times3$ reduced magic squares with magic sum $t$, and $r_\ma(t)$ the number of reduced symmetry types, or equivalently of normalized reduced squares with magic sum $t$.  Then 
\[
M_\aa(t) = \sum_{\substack{0 < s < t \\ s \equiv t \mod{3}}} R_\ma(s) \quad \text{ and } \quad m_\aa(t) = \sum_{\substack{0 < s < t \\ s \equiv t \mod{3}}} r_\ma(s),
\]
since every reduced square with sum $s=t-3k$, where $0 < 3k \leq t-3$, gives one unreduced square with sum $t$ (and positive entries) by adding $3k$ to each entry.  In terms of generating functions,
\begin{equation*}\label{E:redmagica}
\quad \bm_\aa(x) = \frac{x^3}{1-x^3} \,  \br_\ma(x) \,;
\end{equation*}
thus, 
$$
\br_\ma(x) = \frac{ x^{12} (2x^3+1) }{ (1-x^6) (1-x^9) } \,;
$$ 
and $\bR_\ma(x) = 8\br_\ma(x)$.  The quasipolynomial is
\begin{equation}\label{E:mra-constituents}
R_\ma(t) = 8r_\ma(t) = \begin{cases}
\left.\begin{aligned}
&\tfrac43 t-16 = \tfrac43 (t-12),  &&\text{if } t \equiv 0  \\[6pt]
&\tfrac43 t-4 = \tfrac43 (t-3), &&\text{if } t \equiv 3, 15  \\[6pt]
&\tfrac43 t-8 = \tfrac43 (t-6), &&\text{if } t \equiv 6, 12  \\[6pt]
&\tfrac43 t-12 = \tfrac43 (t-9), &&\text{if } t \equiv 9   
\end{aligned}\ \right\} &\text{mod}\,{18}; \\
\\[-8pt]
\ 0,  \qquad\qquad\qquad \quad\quad \text{if } t \not\equiv 0 &\text{mod}\,{18} .
\end{cases}
\end{equation}
The initial nonzero values:
$$
\begin{tabular}{r||c|c|c|c|c|c|c|c|c|c|c|c|c|c|c|c|c|c|cl}
$t$	&12&15&18&21&24&27&30&33&36&39&42&45&48&51&54&57&60&63\\
\hline \hline
$R_\ma(t)$&8&16&8&24&24&24&32&40&32&48&48&48&56&64&56&72&72&72\\
\hline
$r_\ma(t)$&1&2&1&3&3&3&4&5&4&6&6&6&7&8&7&9&9&9\\
\end{tabular}
$$
These sequences are A174256 and A174257 in the OEIS \cite{OEIS}.  

One of the remarkable properties of magic squares of order 3 is that $R_\mc(2k) = R_\ma(3k)$.  The reason is that the middle term of a reduced $3\times3$ magic square equals $s/3$, if $s$ is the magic sum, and the largest entry is $2s/3$.  Thus, the reduced squares of cubic and affine type, allowing for the difference in parameters, are the same, and although the counts of magic squares by magic sum and by upper bound differ, the only reason is that the reduced squares are adjusted differently to get all squares.

The principal constituent $\frac43 t-16$ has constant term whose absolute value is the same as with all other reduced magic quasipolynomials.

We confirmed the constituents by comparing their values to the coefficients of the generating function for several periods.

%%%%%%%%%%%%%%%%%%%%%%%%%%%%%%%%%%%%%%%%%%%%%%%%%%%%%%%%%%%%%%%%%
\section{Semimagic squares of order 3} \label{semimagic}

Now we apply our approach to counting semimagic squares. 
Here is the general form of a reduced, normalized $3\times3$ semimagic square, in which the magic sum is $s = 2\alpha+2\beta+\gamma$:
\begin{equation} 
\label{E:reducednormal}
\begin{tabular}{|c|c|c|}
\hline
\hstrut{6em}	& \hstrut{6em}	& \hstrut{6em}	\\
\raisebox{-5ex}{\vstrut{8ex}}
0			&$\beta$		&$2\alpha+\beta+\gamma$	\\
\hline &&\\
\raisebox{-5ex}{\vstrut{8ex}}
$\alpha+\beta$		&$\alpha+\beta+\gamma-\delta$	&$\delta$	\\
\hline &&\\
\raisebox{-5ex}{\vstrut{8ex}}
$\alpha+\beta+\gamma$	&$\alpha+\delta$	&$\beta-\delta$	\\
\hline
\end{tabular}
\end{equation}

\begin{prop} \label{T:snsemimagic}
A reduced and normalized $3\times3$ semimagic square has the form \eqref{E:reducednormal} with the restrictions
\begin{equation} 
\label{E:reducednormrange}
\begin{aligned}  
0 &< \alpha, \beta, \gamma ;	\\
0 &< \delta < \beta ;
\end{aligned}
\end{equation}
and
\begin{align}  
\label{E:reducednormrestr}
\delta \neq \begin{cases}
  \frac{\beta-\alpha}{2} ,\ \frac{\beta}{2} ,\ \frac{\beta+\gamma}{2} ,\ \frac{\beta+\alpha+\gamma}{2} ; \\
  \beta-\alpha ,\ \alpha+\gamma ;	\\
  \gamma .
\end{cases}
\end{align}
The largest entry in the square is $w := x_{13} = 2\alpha+\beta+\gamma$.

Each reduced normal square with largest entry $w$ corresponds to exactly $72 (t-w-1)$ different magic squares with entries in the range $(0,t)$, for $0<w<t$. Each reduced normal square with magic sum $s$ corresponds to exactly $72$ different magic squares with magic sum equal to $t$, if $t \equiv s \mod{3}$, and none otherwise, for $0<s<t$.
\end{prop}

\begin{proof}
By permuting rows and columns we can arrange that $x_{11} = \min x_{ij}$ and that the top row and left column are increasing.  By flipping the square over the main diagonal we can further force $x_{21}>x_{12}$.  By subtracting the least entry from every entry we ensure that $x_{11}=0$.  Thus we account for the $72(t-w-1)$ semimagic squares that correspond to each reduced normal square.

The form of the top and left sides in \eqref{E:reducednormal} is explained by the fact that $x_{11}<x_{12}<x_{21}<x_{31}<x_{13}$.  The conditions $x_{ij}>x_{11}$ for $i,j=2,3$, together with the row-sum and column-sum equations, imply that $x_{13}$ is the largest entry and that $x_{23}<x_{12}$.  

The only possible equalities amongst the entries are ruled out by the following inequations:
\begin{align*}
x_{22} &\neq x_{12} ,\ x_{21} ,\ x_{23} ;		\\
x_{32} &\neq x_{12} ,\ x_{22} ;		\\
x_{33} &\neq x_{23} ,\ x_{32} .
\end{align*}
These correspond to the restrictions \eqref{E:reducednormrestr}:
\begin{align*}
x_{22} \neq x_{12} &\iff \delta \neq \alpha+\gamma ;		\\
x_{22} \neq x_{21} &\iff \delta \neq \gamma ;			\\
x_{22} \neq x_{23} &\iff \delta \neq \frac{\beta+\alpha+\gamma}{2} ;	\\
x_{32} \neq x_{12} &\iff \delta \neq \beta-\alpha ;		\\
x_{32} \neq x_{22} &\iff \delta \neq \beta+\gamma-\delta \iff \delta \neq \frac{\beta+\gamma}{2} ;	\\
x_{33} \neq x_{23} &\iff \delta \neq \beta-\delta \iff \delta \neq \frac{\beta}{2} ;		\\
x_{33} \neq x_{32} &\iff \delta \neq \alpha+\delta \neq \beta-\delta \iff \delta \neq \frac{\beta-\alpha}{2} . \qedhere
\end{align*}
\end{proof}

%%%%%%%%%%%%%%%%
\subsection{Semimagic squares:  Cubical count (by upper bound)} \label{sc3}

We are counting squares by a strict upper bound on the allowed value of an entry; this bound is the parameter $t$.  Let $S_{\cc}(t)$, for $t>0$, be the number of semimagic squares of order 3 in which every entry belongs to the range $(0,t)$.  

%%%%%%
\subsubsection{Counting the weak squares}
\label{weaksc3}

The polytope $P_\cc$ is the 5-dimensional intersection of $[0,1]^9$ with the semimagic subspace in which all row and column sums are equal.  This polytope is integral because it is the intersection with an integral polytope of a subspace whose constraint matrix is totally unimodular; so the quasipolynomial is a polynomial.  By contrast, the inside-out polytope $(\PH)$ for enumerating strong semimagic squares has denominator 60.  We verified by computer counts for $t \leq 18$ that the weak polynomial is
\begin{align*}
&\frac{3t^5-15t^4+35t^3-45t^2+32t-10}{10} 
= \frac{(t-1)(t^2-2t+2)(3t^2-6t+5)}{10} 
\end{align*}
with generating function
\[
\frac{ (7x^2 - 2x + 1) (2x^3 + x^2 + 4x - 1) }{ (1-x)^6 } \ .
\]
We conclude that $P_\cc$ has volume $3/10$.

%%%%%%
\subsubsection{Reduction of the number of strong squares}
\label{sc3reduction}

We compute $S_\cc(t)$ via $R_\cc(w)$, the number of reduced squares with largest entry equal to $w$.  The formula is
\begin{equation} \label{E:sc3count}
S_\cc(t) = \sum_{w=0}^{t-1} (t-1-w) R_\cc(w) .
\end{equation}
The value $R_\cc(w)=72 r_\cc(w)$ where $r_\cc(w)$ is the number of normal reduced squares (in which we know the largest entry to be $x_{13}$).  Thus $r_\cc(s)$ counts the number of $\frac{1}{s}$-integral points in the interior of the 3-dimensional polytope $Q_\cc$ defined by
\begin{equation} \label{E:sc3reducednormpoly}
0 \leq x, y ;  \qquad  0 \leq z \leq y ;  \qquad  x + y \leq \frac12 
\end{equation}
with the seven excluded (hyper)planes
\begin{equation} \label{E:sc3arr}
 z = \frac{y-x}{2} ,\ \frac{y}{2} ,\ \frac{1-y-2x}{2} ,\ \frac{1-x-y}{2} ,\ y-x ,\ 1-x-2y ,\ 1-2x-2y \ ,
\end{equation}
the three coordinates being $x=\alpha/w$, $y=\beta/w$, and $z=\delta/w$. 

The hyperplane arrangement for reduced normal squares is that of \eqref{E:sc3arr}.  We call it $\cI_\cc$.  Thus $r_\cc(s) = E^\circ_{Q_\cc^\circ,\cI_\cc}(s)$.

%%%%%%
\subsubsection{Geometrical analysis of the reduced normal polytope}
\label{sc3geometry}

We apply M\"obius inversion, Equation \eqref{E:oehrhypgf}, over the intersection poset $\cL(Q_\cc^\circ,\cI_\cc)$. 
We need to know not only $\cL(Q_\cc^\circ,\cI_\cc)$ but also all the vertices of $(Q_\cc,\cI_\cc)$, since they are required for computing the Ehrhart generating function and estimating the period of the reduced normal quasipolynomial.

We number the planes:
\begin{alignat*}{2}
&\pi_1: \quad&	x-y+2z &= 0 , \\
&\pi_2: & 	y-2z &= 0 , \\
&\pi_3: & 	2x+2z &= 1 , \\
&\pi_4: & 	x+2z &= 1 , \\
&\pi_5: & 	x-y+z &= 0 , \\
&\pi_6: & 	x+y+z &= 1 , \\
&\pi_7: & 	2x+y+z &= 1 .
\end{alignat*}
The intersection of two planes, $\pi_j \cap \pi_k$, is a line we call $l_{jk}$; $\pi_3 \cap \pi_5 \cap \pi_6$ is a line we also call $l_{356}$.  The intersection of three planes is, in general, a point but not usually a 
vertex of $(Q_\cc,\cI_\cc)$.

Our notation for the line segment with endpoints $X,Y$ is $XY$, while $\overline{XY}$ denotes the entire line spanned by the points.  The triangular convex hull of three noncollinear points $X,Y,Z$ is $XYZ$.  We do not need quadrilaterals, as the intersection of each plane with $Q$ is a triangle.

We need to find the intersections of the planes with $Q_\cc^\circ$, separately and in combination.  Here is a list of significant points; we shall see it is the list of vertices of $(Q_\cc,\cI_\cc)$.  The first column has the vertices of $Q_\cc$, the second the vertices of $(Q_\cc,\cI_\cc)$ that lie in open edges, the third the vertices that lie in open facets, and the last is the sole interior vertex.
\begin{equation}
\label{E:sc3snvertices}
\begin{aligned}
O&=(0,0,0) , \quad&	D_\cc&=(0,\tfrac12,\tfrac12)\in OC , \quad&	F_\cc&=(0,\tfrac23,\tfrac13)\in OBC , \quad& H_\cc&=(\tfrac15,\tfrac25,\tfrac15).\\
A&=(\tfrac12,0,0) , &	E_\cc&=(\tfrac13,\tfrac13,0)\in AB , 	&	G_\cc&=(\tfrac14,\tfrac12,\tfrac14)\in ABC , & \\
B_\cc&=(0,1,0) , &	E_\cc'&=(\tfrac13,\tfrac13,\tfrac13)\in AC , &	G_\cc'&=(\tfrac15,\tfrac35,\tfrac15)\in ABC , & \\
C_\cc&=(0,1,1) , &	E_\cc''&=(0,1,\tfrac12)\in BC , & 		G_\cc''&=(\tfrac15,\tfrac35,\tfrac25)\in ABC , &
\end{aligned}
\end{equation}
The denominator of $(Q_\cc,\cI_\cc)$ is the least common denominator of all the points; it evidently equals $4\cdot3\cdot5 = 60$.

The intersections of the planes with the edges of $Q_\cc$ are in Table \ref{Tb:planeedge}.  The subscript $\cc$ is omitted.

\begin{table}
\renewcommand\arraystretch{1.3}
\begin{tabular}{|l||c|c|c|c|c|c||c|}
\hline
Plane	&\multicolumn{6}{c||}{Intersection with edge}	&Intersection	\\
	&$OA$	&$OB$	&$OC$	&$AB$	&$AC$	&$BC$	&with $Q$	\\
\hline\hline
$\pi_1$	&$O$	&$O$	&$O$	&$E$	&$\notin Q$&$E''$&$OEE''$	\\
\hline
$\pi_2$	&$OA$	&$O$	&$O$	&$A$	&$A$	&$E''$	&$OAE''$	\\
\hline
$\pi_3$	&$A$	&$\eset$&$D$	&$A$	&$A$	&$E''$	&$ADE''$	\\
\hline
$\pi_4$	&$\notin Q$&$\eset$&$D$&$\notin Q$&$E'$&$E''$&$DE'' E'$	\\
\hline
$\pi_5$	&$O$	&$O$	&$OC$	&$E$	&$C$	&$C$	&$OCE$	\\
\hline
$\pi_6$	&$\notin Q$&$B$	&$D$	&$B$	&$E'$	&$B$	&$BD E'$	\\
\hline
$\pi_7$	&$A$	&$B$	&$D$	&$AB$	&$A$	&$B$	&$ABD$	\\
\hline
\end{tabular}
\vskip 15pt
\caption{Intersections of planes of $\cI$ with edges of $Q$.  A vertex of $Q$ contained in $\pi_j$ will show up three times in the row of $\pi_j$.  In order to clarify the geometry, we distinguish between a plane's meeting an edge line outside $Q$ and not meeting it at all (i.e., their being parallel).}
\label{Tb:planeedge}
\end{table}

Table \ref{Tb:cubiclines} shows the lines generated by pairwise intersection of planes.

\begin{table}[hb]
\renewcommand\arraystretch{1.2}
\begin{tabular}{|l||c|c|c|c|c|c|}
\hline
	&$\pi_2$&$\pi_3$&$\pi_4$&$\pi_5$&$\pi_6$&$\pi_7$	\\
\hline\hline
$\pi_1$	&\lineeq{x}{0}{y}{2z}	&\lineeq{x+z}{\frac12}{y-z}{\frac12}	&\lineeq{y}{1}{x+2z}{1}	&\lineeq{z}{0}{x}{y}	&\lineeq{z}{2y-1}{x}{2-3y}	&\lineeq{x+z}{\frac13}{y-z}{\frac13}	\\
\hline
$\pi_2$	&	&\lineeq{y}{2z}{x+z}{\frac12}	&\lineeq{y}{2z}{x+y}{1}	&\lineeq{y}{2z}{x}{z}	&\lineeq{y}{2z}{x+3z}{1}	&\lineeq{y}{2z}{2x+3z}{1}	\\
\hline
$\pi_3$	&&	&\lineeq{x}{0}{z}{\frac12}	&\multicolumn{2}{c|}{$l_{356}$:\qquad\lineeq{x+z}{\frac12}{y}{\frac12}}		&\lineeq{x+z}{\frac12}{y}{z}	\\
\hline
$\pi_4$	&&&	&\lineeq{x}{1-2z}{y}{1-z}	&\lineeq{x}{1-2z}{y}{z}	&\lineeq{x}{1-2z}{y}{3z-1}{}{}{}{}	\\
\hline
$\pi_5$	&&&&	&$l_{356}$	&\lineeq{x}{1-2y}{z}{3y-1}	\\
\hline
$\pi_6$	&&&&&	&\lineeq{x}{0}{y+z}{1}	\\
\hline
\end{tabular}
\vskip 15pt
\caption{The equations of the pairwise intersections of planes of $\cI_\cc$.}
\label{Tb:cubiclines}
\end{table}

In Table \ref{Tb:linepolytope} we describe the intersection of each line 
with $Q_\cc$ and with its interior.  The subscript $\cc$ is omitted.

\begin{table}[ht]
\renewcommand\arraystretch{1.2}
\begin{tabular}{|l||c|c|c|c|c|c|}
\hline
	&$\pi_2$&$\pi_3$&$\pi_4$&$\pi_5$&$\pi_6$&$\pi_7$	\\
\hline\hline
$\pi_1$	&\intersects{OE''}{(OBC)}	&\intersects{E''}{(BC)}	&\intersects{E''}{(BC)}	&\intersects{OE}{(OAB)}	&$FG'$	&$EF$	\\
\hline
$\pi_2$	&	&\intersects{AE''}{(ABC)}	&\intersects{E''}{(BC)}	&$OG$	&$FG$	&$AF$	\\
\hline
$\pi_3$	&&	&\intersects{DE''}{(OBC)}	&\multicolumn{2}{c|}{$l_{356}$:\quad$DG$}	&\intersects{AD}{(OAC)}	\\
\hline
$\pi_4$	&&&	&$DG''$	&\intersects{D E'}{(OAC)}	&\intersects{D}{(OC)}	\\
\hline
$\pi_5$	&\intersects{}{}&&&	&$l_{356}$	&$ED$	\\
\hline
$\pi_6$	&&&&&	&\intersects{BD}{(OBC)}	\\
\hline
\end{tabular}
\vskip 15pt
\caption{The intersections of lines with $Q$ and $Q^\circ$.  The second (parenthesized) row in each box shows the smallest face of $Q$ to which the intersection belongs, if that is not $Q$ itself; these intersections are not part of the intersection poset of $(Q^\circ,\cI)$.}
\label{Tb:linepolytope}
\end{table}

Last, we need the intersection points of three planes of $\cI_\cc$; or, of a plane and a line.  Some are not in $Q_\cc$ at all; them we can ignore.  Some are on the boundary of $Q_\cc$; they are necessary in finding the denominator, but all of them are points already listed in \eqref{E:sc3snvertices}.  It turns out that 
$$
\pi_2 \cap \pi_5 \cap \pi_7 = H_\cc 
$$
is the only vertex in $Q_\cc^\circ$, so it is the only one we need for the intersection poset.

Here, then, is the intersection poset (Figure \ref{F:s3intersections}).  The subscript $\cc$ is omitted.  In the figure, for simplicity, we write $\pi_j$, etc., when the actual element is the simplex $\pi_j \cap Q^\circ$, etc.; we also state the vertices of the simplex.  The M\"obius function $\mu(\0,u)$ equals $(-1)^{\codim u}$ with the exception of $\mu(\0,l_{356}) = 2$.

\begin{figure}[t]
\begin{center}
\psfrag{L}[c]{{\large$\cL(Q^\circ,\cI)$}}
\psfrag{Q}[c]{$\bbR^3\ (OABC)$}
\psfrag{pi1}[c]{$\pi_1\ (OEE'')$}
\psfrag{pi2}[c]{$\pi_2\ (OAE'')$}
\psfrag{pi3}[c]{$\pi_3\ (ADE'')$}
\psfrag{pi4}[c]{$\pi_4\ (DE'' E')$}
\psfrag{pi5}[c]{$\pi_5\ (OCE)$}
\psfrag{pi6}[c]{$\pi_6\ (BD E')$}
\psfrag{pi7}[c]{$\pi_7\ (ABD)$}
\psfrag{l16}[c]{$l_{16}\ (FG')$}
\psfrag{l17}[c]{$l_{17}\ (EF)$}
\psfrag{l356}[c]{$l_{356}\ (DG)$}
\psfrag{l57}[c]{$l_{57}\ (ED)$}
\psfrag{l25}[c]{$l_{25}\ (OG)$}
\psfrag{l26}[c]{$l_{26}\ (FG)$}
\psfrag{l27}[c]{$l_{27}\ (AF)$}
\psfrag{l45}[c]{$l_{45}\ (DG'')$}
\psfrag{G5}[c]{$H$}
\includegraphics{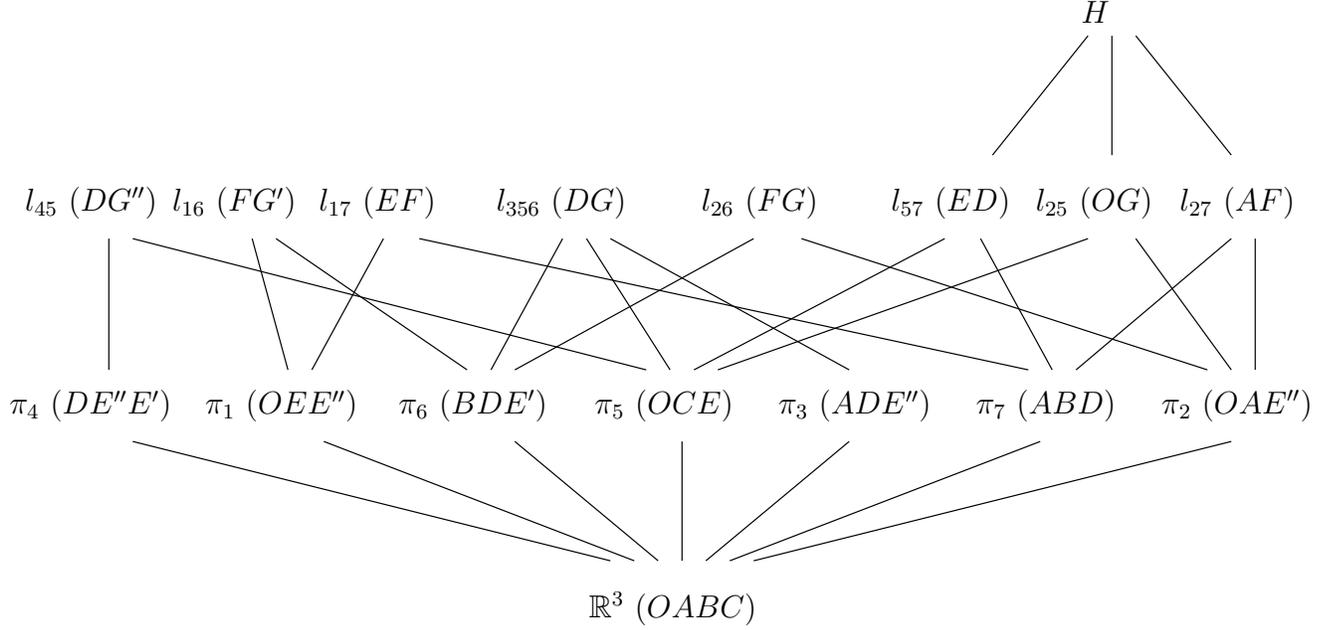}
\caption{The intersection poset $\cL(Q^\circ,\cI)$ for semimagic squares.  The diagram shows both the flats and (in parentheses) their intersections with $Q$.}
\label{F:s3intersections}
\end{center}
\end{figure}
%

%%%%%%
\subsubsection{Generating functions and the quasipolynomial}
\label{sc3gf}

That was the first half of the work.  The second half begins with finding $r_\cc(w)=E^\circ_{(Q_\cc^\circ,\cI_\cc)}(w)$ from the Ehrhart generating functions $\bE_{u}(t)$ of the intersections by means of \eqref{E:gfmobius}. The next step, then, is to calculate all necessary generating functions.  This is done by {\tt LattE}. $\br_\cc(x)$ is the result of applying reciprocity to the sum of all these rational functions (with the appropriate M\"obius-function multiplier $-1$, excepting $\bE_{DG^\circ}(x)$ whose multiplier is $-2$). The result is:
\begin{equation}
\label{E:sc3sngf}
\begin{aligned}
-\br_\cc(1/x) =\ 
&\bE_{ OABC } (x) + \bE_{OEE''}(x) + \bE_{OAE''}(x) \\
&+ \bE_{ADE''}(x) + \bE_{DE'E''}(x) + \bE_{OCE}(x) \\
&+ \bE_{BDE'}(x) + \bE_{ABD}(x) + \bE_{FG'}(x) \\
&+ \bE_{EF}(x) + \bE_{OG}(x) + \bE_{FG}(x) \\
&+ \bE_{AF}(x) + 2\bE_{DG}(x) + \bE_{DG''}(x) \\
&+ \bE_{DE}(x) + \bE_{H}(x)	\\
 =\ &\frac{ 1 }{ (1-x)^3 (1-x^2) } 
+ \frac{1}{ (1-x) (1-x^2) (1-x^3) } 
+ \frac{1}{ (1-x) (1-x^2)^2 }	\\
&+ \frac{1}{ (1-x^2)^3 } 
+ \frac{1}{ (1-x^2)^2 (1-x^3) } 
+ \frac{1}{ (1-x)^2 (1-x^3) }	\\
&+ \frac{1}{ (1-x) (1-x^2) (1-x^3) } 
+ \frac{1}{ (1-x) (1-x^2)^2 } 
+ \frac{1}{ (1-x^3) (1-x^5) }	\\
&+ \frac{1}{ (1-x^3)^2 } 
+ \frac{1}{ (1-x) (1-x^4) } 
+ \frac{1}{ (1-x^3) (1-x^4) }	\\
&+ \frac{1}{ (1-x^2) (1-x^3) } 
+ 2\frac{1}{ (1-x^2) (1-x^4) } 
+ \frac{1}{ (1-x^2) (1-x^5) }	\\
&+ \frac{1}{ (1-x^2) (1-x^3) } 
+ \frac{1}{1-x^5} \ . 
\end{aligned}
\end{equation}
Then by \eqref{E:sc3count} the generating function for cubically counted semimagic squares is
\begin{equation}\label{E:sc3gf}
\begin{aligned}
\bS_\cc(x) &= 72 \frac{x^2}{(1-x)^2} \br_\cc(x) \\[5pt]
&= \frac{ 
72 x^{10} \left[ 18\,x^{9}+46\,x^{8}+69\,x^{7}+74\,x^{6}+65\,x^{5}+46\,x^{4}+26\,x^3+11\,x^2+4\,x+1 \right] 
}{ 
(1-x^2)^{2} (1-x^{3})^{2} (1-x^{4})  (1-x^{5}) }\ .
\end{aligned}
\end{equation}

 From the geometrical denominator 60 or the (standard-form) algebraic denominator $(1-x^{60})^{5}$ we know the period of $S_\cc(t)$ divides 60.  We compute the constituents of $S_\cc(t)$ by the method of Section \ref{genmethod}; the result is that
\begin{equation}\label{E:sc-constituents}
S_{\cc}(t) = \begin{cases}
\ds \frac{3}{10}  t^5  -\frac{75}{8}  t^4  + \frac{331}{3}  t^3  -\frac{5989}{10}  t^2  + c_1(t) t  - c_0(t),	&\text{if $t$ is even}; \\ \\
\ds \frac{3}{10}  t^5  -\frac{75}{8}  t^4  + \frac{331}{3}  t^3  -\frac{11933}{20}  t^2  + c_1(t) t  - c_0(t),	&\text{if $t$ is odd};
\end{cases}
\end{equation}
where $c_1$ varies with period 6, given by 
$$
c_1(t) = \begin{cases}
1464,		&\text{ if } t \equiv 0, 2  \\
1456,		&\text{ if } t \equiv 4  \\
\frac{2831}{2}\,,	&\text{ if } t \equiv 1  \\
\frac{2847}{2}\,,	&\text{ if } t \equiv 3, 5  \\
\end{cases} \qquad \mod{6};
$$
and $c_0$, given by Table \ref{Tb:semicubc0}, varies with period 60.  (It is curious that the even constant terms have half the period of the odd terms.)  Thus the period of $S_{\cc}$ turns out to be 60, the largest it could be.

\begin{table}[htb]
\newcommand\Strut{\rule{0ex}{2.5ex}}
\begin{tabular}{|r|c||r|c||r|c||r|c||r|c|}
\hline
\Strut $t$	& $c_0(t)$	&$t$	& $c_0(t)$	&$t$	& $c_0(t)$	&$t$	& $c_0(t)$	&$t$	& $c_0(t)$	\\[3pt]
\hline\hline
\Strut 0&$1296$			&12&$1296$		&24&$\frac{6624}{5}$	&36&$\frac{6192}{5}$	&48&$\frac{6624}{5}$	\\[3pt]
\hline
\Strut 1&$\frac{110413}{120}$	&13&$\frac{120781}{120}$&25&$\frac{23465}{24}$	&37&$\frac{23465}{24}$	&49&$\frac{120781}{120}$ \\[3pt]
\hline
\Strut 2&$\frac{3824}{3}$	&14&$\frac{19552}{15}$	&26&$\frac{18256}{15}$	&38&$\frac{19552}{15}$	&50&$\frac{3824}{3}$	\\[3pt]
\hline
\Strut 3&$\frac{47727}{40}$	&15&$\frac{9315}{8}$	&27&$\frac{9315}{8}$	&39&$\frac{47727}{40}$	&51&$\frac{44271}{40}$	\\[3pt]
\hline
\Strut 4&$\frac{18152}{15}$	&16&$\frac{16856}{15}$	&28&$\frac{18152}{15}$	&40&$\frac{3544}{3}$	&52&$\frac{3544}{3}$	\\[3pt]
\hline
\Strut 5&$\frac{25705}{24}$	&17&$\frac{25705}{24}$	&29&$\frac{131981}{120}$&41&$\frac{121613}{120}$&53&$\frac{131981}{120}$ \\[3pt]
\hline
\Strut 6&$\frac{6192}{5}$	&18&$\frac{6624}{5}$	&30&$1296$		&42&$1296$		&54&$\frac{6624}{5}$	\\[3pt]
\hline
\Strut 7&$\frac{25193}{24}$	&19&$\frac{129421}{120}$&31&$\frac{119053}{120}$&43&$\frac{129421}{120}$&55&$\frac{25193}{24}$	\\[3pt]
\hline
\Strut 8&$\frac{19552}{15}$	&20&$\frac{3824}{3}$	&32&$\frac{3824}{3}$	&44&$\frac{19552}{15}$	&56&$\frac{18256}{15}$	\\[3pt]
\hline
\Strut 9&$\frac{44847}{40}$	&21&$\frac{41391}{40}$	&33&$\frac{44847}{40}$	&45&$\frac{8739}{8}$	&57&$\frac{8739}{8}$	\\[3pt]
\hline
\Strut 10&$\frac{3544}{3}$	&22&$\frac{3544}{3}$	&34&$\frac{18152}{15}$	&46&$\frac{16856}{15}$	&58&$\frac{18152}{15}$	\\[3pt]
\hline
\Strut 11&$\frac{130253}{120}	$&23&$\frac{140621}{120}$&35&$\frac{27433}{24}$	&47&$\frac{27433}{24}$	&59&$\frac{140621}{120}$ \\[3pt]
\hline
\end{tabular}
\vspace{15pt}
\caption{Constant terms (without the negative sign) of the constituents 
of the semimagic cubical quasipolynomial $S_\cc(t)$.}
\label{Tb:semicubc0}
\end{table}

The principal constituent (for $t\equiv0$) is
$$
\frac{3}{10} t^5 - \frac{75}{8} t^4 + \frac{331}{3} t^3 - \frac{5989}{10} t^2 + 1464 t - 1296 .
$$
Its unsigned constant term, 1296, is the number of order types of semimagic squares.  Allowing for the 72 symmetries of a semimagic square, there are just 18 symmetry classes of order types.

For the first few nonzero values of $S_\cc(t)$ see the following table.  (This is sequence A173546 in the OEIS \cite{OEIS}.)  The third row is the number of normalized squares, or symmetry classes (sequence A173723), which equals $S_\cc(t)/72$.  The other lines give the numbers of reduced squares (sequence A173727) and of reduced normal squares (i.e., symmetry types of reduced squares; sequence A173724), which may be of interest.
$$
\begin{tabular}{r||c|c|c|c|c|c|c|c|c|c|c|c|c|c|c|c|}
$t$     	&8&9&10&11&12&13&14&15&16&17&18&19\\
\hline \hline
$S_\cc(t)$&0&0&72&288&936&2592&5760&11520&20952&35712&57168&88272\\
\hline
$\hS_\cc(t)$&0&0&1&4&13&36&80&160&291&496&794&1226\\
\hline
$R_\cc(t)$&72&144&432&1008&1512&2592&3672&5328&6696&9648&11736&15552\\
\hline
$r_\cc(t)$	&1&2&6&14&21&36&51&74&93&134&163&216\\
\end{tabular}
$$
%{Values of $S,\hS,n$ from LattE/Maple.}

Compare the strong to the weak quasipolynomial.  The leading coefficients agree and the strong coefficient of $t^4$ is constant. These facts, of which we made no use in deducing the quasipolynomial, provide additional verification of the correctness of the counts and constituents.

%%%%%%
\subsubsection{Another method:  Direct counting}
\label{sc3alt}

We checked the constituents by directly counting (in Maple) all semimagic squares for $t \leq 100$.  The numbers agreed with those derived from the generating function and quasipolynomial above.

%%%%%%%%%%%%%%%%
\subsection{Semimagic squares:  Affine count (by magic sum)} \label{sa3}

Now we count squares by magic sum: we compute $S_\aa(t)$, the number of squares with magic sum $t$.

%%%%%%
\subsubsection{The Birkhoff polytope}\label{sa3birk}

The polytope $P$ for semimagic squares of order 3, counted by magic sum, is 4-dimensional and integral.  (It is the polytope of doubly stochastic matrices of order 3, i.e., a Birkhoff polytope \cite{BvN,BP}.) 

%%%%%%
\subsubsection{Affine weak semimagic}\label{sa3weak}

The polytope for weak semimagic squares of order 3 is the same $P$.  

The weak quasipolynomial, or rather, polynomial, first computed by MacMahon \cite[Vol.\ II, par.\ 407, p.\ 161]{MacM}, is
$$
\frac{t^4-6t^3+15t^2-18t+8}{8} = \frac{(t-1)(t-2)(t^2-3t+4)}{8} 
$$  
with generating function
\[
\frac{ 6x^4 - 9x^3 + 10x^2 - 5x + 1 }{ (1-x)^5 } \ .
\]

%%%%%%
\subsubsection{Reduction}\label{sa3reduction}

The count is via $R_\aa(s)$, the number of reduced squares with magic sum $s$.  The formula is
\begin{align} 
\label{E:sa3count}
S_\aa(t) &= \sum_{\substack{0 < s \leq t-3 \\ s \equiv t \mod{3}}} R_\aa(s) &\text{ if } t>0 .
\end{align}
We have $R_\aa(s)=72 r_\aa(w)$, where $r_\aa(s)$ is the number of reduced, normalized squares with magic sum $s$, equivalently the number of $\frac{1}{s}$-integral points in the interior of the 3-dimensional polytope $Q_\aa$ defined by
\begin{equation}\label{E:sa3reducednormpoly}
0 \leq x, y ;  \qquad  0 \leq z \leq y ;  \qquad  x + y \leq \frac12 
\end{equation}
with the seven excluded (hyper)planes
\begin{equation}\label{E:sa3arr}
z = \frac{y-x}{2} ,\ \frac{y}{2} ,\ \frac{1-y-2x}{2} ,\ \frac{1-x-y}{2} ,\ 
y-x ,\ 1-x-2y ,\ 1-2x-2y \ ,
\end{equation}
the three coordinates being $x=\alpha/s$, $y=\beta/s$, and $z=\delta/s$.  

The hyperplane arrangement for reduced, normalized squares is that of \eqref{E:sa3arr}.  We call it $\cI_\aa$.  Thus $r_\aa(s) = E^\circ_{Q_\aa^\circ,\cI_\aa}(s)$.

%%%%%%
\subsubsection{The reduced, normalized weak polytopal quasipolynomial} 
\label{weaksa3}

This function simply counts $\frac1s$-lattice points in $Q_\aa^\circ$.  The counting formula is
$
\sum_\alpha \sum_\beta \sum_{\delta} 1 ,
$
summed over all triples that satisfy \eqref{E:reducednormrange}.  It simplifies to
$$
\sum_\alpha \binom{\lfloor\frac{s-1}{2}\rfloor-\alpha}{2} ,
$$
which gives the Ehrhart quasipolynomial
\begin{equation}\label{E:sa3snordinary}
E_{Q_\aa^\circ}(s) = \binom{\lfloor\frac{s-1}{2}\rfloor}{3} = \begin{cases}
\dbinom{\frac{s-1}{2}}{3} = \frac{1}{48}(s-1)(s-3)(s-5),  &\text{for odd } s ; \\ \\
\dbinom{\frac{s-2}{2}}{3} = \frac{1}{48}(s-2)(s-4)(s-6),  &\text{for even } s .
\end{cases}
\end{equation}
The leading coefficient is $\vol Q_\aa$.

We deduce from \eqref{E:sa3snordinary} that 
$$
\bE_{P^\circ}(x) = \frac{x^7(1+x)}{(1-x^2)^4}
$$
and by reciprocity that
$$
\bE_P(x) = \frac{1+x}{(1-x^2)^4} .
$$

%%%%%%
\subsubsection{Geometrical analysis of the reduced, normalized polytope}
\label{sa3geometry}

We apply M\"obius inversion, Equation \eqref{E:gfmobius}, over the intersection poset $\cL(Q_\aa^\circ,\cI)$.

We number the planes:
\begin{alignat*}{2}
&\pi_1: \quad&	x-y+2z &= 0 , \\
&\pi_2: & 	y-2z &= 0 , \\
&\pi_3: & 	2x+y+2z &= 1 , \\
&\pi_4: & 	x+y+2z &= 1 , \\
&\pi_5: & 	x-y+z &= 0 , \\
&\pi_6: & 	x+2y+z &= 1 , \\
&\pi_7: & 	2x+2y+z &= 1 .
\end{alignat*}
The intersection of two planes, $\pi_j \cap \pi_k$, is a line we call $l_{jk}$; $\pi_3 \cap \pi_5 \cap \pi_6$ is a line we also call $l_{356}$.  The intersection of three planes is, in general, a point but not usually a vertex of $(Q_\aa,\cI_\aa)$. Our geometrical notation is as in the cubical analysis.

We need to find the intersections of the planes with $Q_\aa^\circ$, separately and in combination.  Here is a list of significant points; we shall see it is the list of vertices of $(Q_\aa,\cI_\aa)$.  The first column has the vertices of $Q_\aa$, the second the vertices of $(Q_\aa,\cI_\aa)$ that lie in open edges, the third the vertices that lie in open facets, and the last is the sole interior vertex.
\begin{equation}\label{E:sn3vertices}
\begin{aligned}
O&=(0,0,0), \quad &		D_\aa&=(0,\tfrac13,\tfrac13)\in OC, \quad&	F_\aa&=(0,\tfrac25,\tfrac15)\in OBC, \ \ & H_\aa&=(\tfrac17,\tfrac27,\tfrac17),\\
A&=(\tfrac12,0,0), &		E_\aa&=(\tfrac14,\tfrac14,0)\in AB, & 		G_\aa&=(\tfrac16,\tfrac26,\tfrac16)\in ABC, \\
B_\aa&=(0,\tfrac12,0), &		E_\aa'&=(\tfrac14,\tfrac14,\tfrac14)\in AC, &	G_\aa'&=(\tfrac18,\tfrac38,\tfrac18)\in ABC, \\
C_\aa&=(0,\tfrac12,\tfrac12), &	E_\aa''&=(0,\tfrac12,\tfrac14)\in BC, &		G_\aa''&=(\tfrac18,\tfrac38,\tfrac28)\in ABC.
\end{aligned}
\end{equation}
The least common denominator of $O,A,B_\aa,C_\aa$ explains the period 2 of $E_{Q_\aa^\circ}$.  The denominator of $(Q_\aa,\cI_\aa)$ is the least common denominator of all the points; it evidently equals 
$8\cdot3\cdot5\cdot7 = 840$.

The intersections of the planes with the edges of $Q_\aa$ are in Table \ref{Tb:planeedge}.  Table \ref{Tb:affinelines} shows the lines generated by pairwise intersection of planes.  Table \ref{Tb:linepolytope} describes the intersection of each line with $Q_\aa$ and with $Q_\aa^\circ$.

\begin{table}[h]
\renewcommand\arraystretch{1.3}
\begin{tabular}{|l||c|c|c|c|c|c|}
\hline
	&$\pi_2$&$\pi_3$&$\pi_4$&$\pi_5$&$\pi_6$&$\pi_7$	\\
\hline\hline
$\pi_1$	&\lineeq{x}{0}{y}{2z}	&\lineeq{x}{\tfrac{1-4z}3}{y}{\tfrac{1+2z}3}	&\lineeq{x+2z}{\frac12}{y}{\frac12}	&\lineeq{x}{y}{z}{0}	&\lineeq{x}{2-5y}{z}{3y-1}	&\lineeq{x}{\tfrac{1-5z}4}{y}{\tfrac{1+3z}4}	\\
\hline
$\pi_2$	&	&\lineeq{x+y}{\frac12}{y}{2z}	&\lineeq{x}{1-2y}{y}{2z}	&\lineeq{x}{z}{y}{2z}	&\lineeq{x}{1-5z}{y}{2z}	&\lineeq{x}{\frac{1-5z}2}{y}{2z}	\\
\hline
$\pi_3$	&&	&\lineeq{x}{0}{y}{1-2z}	&\multicolumn{2}{c|}{$l_{356}$:\qquad\lineeq{x+z}{\frac13}{y}{\frac13}}		&\lineeq{2x+3y}{1}{z}{y}	\\
\hline
$\pi_4$	&&&	&\lineeq{x}{3y-1}{z}{1-2y}	&\lineeq{x}{1-3y}{z}{y}	&\lineeq{x+y}{\frac13}{z}{\frac13}	\\
\hline
$\pi_5$	&&&&	&$l_{356}$	&\lineeq{x}{1-3y}{z}{4y-1}	\\
\hline
$\pi_6$	&&&&&	&\lineeq{x}{0}{z}{1-2y}	\\
\hline
\end{tabular}
\vskip 15pt
\caption{The equations of the pairwise intersections of planes of $\cI_\aa$.}
\label{Tb:affinelines}
\end{table}

Last, we need the intersection points of three planes of $\cI_\aa$; or, of a plane and a line.  Some are not in $Q_\aa$ at all; them we can ignore.  Some are on the boundary of $Q_\aa$; they are necessary in finding the denominator, but all of them are points already listed in \eqref{E:sn3vertices}.  It turns out that
$$
\pi_2 \cap \pi_5 \cap \pi_7 = H_\aa 
$$
is the only vertex in $Q_\aa^\circ$, so it is the only one we need for the intersection poset.

The combinatorial structure and the intersection poset (Figure \ref{F:s3intersections}) for the affine count are identical to those for the cubical count.  The reason is that the affine polytope $P_\aa$ is the 4-dimensional section of $P_\cc$ by the flat in which the magic sum equals 1, and this flat is orthogonal to the line of intersection of the whole arrangement $\cH_\aa$.

%%%%%%
\subsubsection{Generating functions and the quasipolynomial}
\label{sa3gf}

The second half of the affine solution is to find $\br_\aa(s)=\bE^\circ_{(Q_\aa^\circ,\cI_\aa)}(s)$ by applying Equations \eqref{E:ehrrecipgf}--\eqref{E:ehrhypgf} after finding the Ehrhart generating functions $\bE_{u}(s)$ for $u \in \cL(Q_\aa^\circ,\cI_\aa)$. The next step, then, is to calculate those generating functions.  This is done by {\tt LattE}.  Then $(-1)^3 \br_\aa(x\inv)$ is the sum of all these rational functions; that is,
\allowdisplaybreaks
\begin{equation}
\label{E:sa3sngfsum}
\begin{aligned}
- \br_\aa(1/x) =\ &\bE_{OABC}(x) + \bE_{OEE''}(x) + \bE_{OAE''}(x) \\
&+ \bE_{ADE''}(x) + \bE_{DE'' E'}(x) + \bE_{OC E}(x) \\
&+ \bE_{B D E'}(x) + \bE_{AB D}(x) + \bE_{F G'}(x) \\
&+ \bE_{E F}(x) + \bE_{OG}(x) + \bE_{F G}(x) \\
&+ \bE_{AF}(x) + 2 \bE_{D G}(x) + \bE_{D G''}(x) \\
&+ \bE_{D E}(x) + \bE_{H}(x) \\
=\ &\frac{1}{(1-x) (1-x^2)^3} + \frac{1}{(1-x) (1-x^4)^2} + \frac{1}{(1-x) (1-x^2) (1-x^4)} \\
&+ \frac{1}{(1-x^2) (1-x^3) (1-x^4)} + \frac{1}{(1-x^3) (1-x^4)^2} + \frac{1}{(1-x) (1-x^2) (1-x^4)} \\
&+ \frac{1}{(1-x^2) (1-x^3) (1-x^4)} + \frac{1}{(1-x^2)^2 (1-x^3)} + \frac{1}{(1-x^5) (1-x^8)} \\
&+ \frac{1}{(1-x^4) (1-x^5)} + \frac{1}{(1-x) (1-x^6)} + \frac{1}{(1-x^5) (1-x^6)} \\
&+ \frac{1}{(1-x^2) (1-x^5)} + 2 \frac{1}{(1-x^3) (1-x^6)} + \frac{1}{(1-x^3) (1-x^8)} \\
&+ \frac{1}{(1-x^3) (1-x^4)} + \frac{1}{1-x^7} \ .
\end{aligned}
\end{equation}
The generating function for the affine count of semimagic squares, by \eqref{E:sa3count}, is
\begin{equation}
\label{E:sa3gf}
\begin{aligned}
 &\quad \bS_\aa(x) = 72 \, \frac{x^3}{1-x^3} \, \br_\aa(x) = \\[5pt]
 &\frac{ 
 72x^{15}\, \left\{  \begin{aligned}
18x^{21}&+5x^{20}+15x^{19}+11x^{17}-8x^{16}+x^{15}-23x^{14}-13x^{13}-22x^{12}-9x^{11}\\
&-16x^{10}+x^{9}-3x^{8}+7x^{7}+7x^{6}+9x^{5}+7x^{4}+6x^{3}+4x^{2}+2x+1
 \end{aligned} \right\}
 }{ 
 (1-x^3)^2 (1-x^4) (1-x^5) (1-x^6) (1-x^7) (1-x^8)
 } \ .
\end{aligned}
\end{equation}

 From the geometrical or generating-function denominator we know that the period of $S_\aa(t)$ divides $840=\lcm(3,4,6,7,8)$.  This is long, but it can be simplified.  The factor 7 in the period is due to a single term in \eqref{E:sa3sngfsum}.  If we treat it separately we have $\br_\aa$ as a sum of the $H$-term $x^7/(1-x^7)$ and a ``truncated'' generating function for $\br_\aa(x) + \frac{x^7}{1-x^7}$, and a corresponding truncated expression 
\begin{align*}
&\bS_\aa(x) - 72 \frac{x^{10}}{(1-x^3)(1-x^7)} =  \\[5pt]
&\quad \frac{ 
-72x^{10}\, \left\{  \begin{aligned}
17x^{19}&+5x^{18}+15x^{17}+x^{16}+12x^{15}-7x^{14}+2x^{13}-7x^{12}\\
&-8x^{11}-9x^{10}-9x^{9}-6x^{8}-6x^{7}-x^{6}+x^{4}+x^{3}-1
\end{aligned} \right\}
}{ 
(1-x^3)^2 (1-x^4) (1-x^5) (1-x^6) (1-x^8)
} \ .
\end{align*}
We extract the constituents from this expression as in Section \ref{genmethod}, separately for the two parts of the generating function.  The constituents are all of the form
\begin{equation}\label{E:sa-constituents}
S_{\aa}(t) = \frac{1}{8} t^4 - \frac{9}{2} t^3 + a_2(t) t^2 - a_1(t) t + a_0(t) - 72 S_7(t) ,
\end{equation}
where $S_7(t)$ is a correction, to be defined in a moment, and 
\[
a_2(t) = \begin{cases}
\frac{243}{4}\,,  &\text{if } t \equiv 0  \\ 
\frac{218}{4}\,,  &\text{if } t \equiv 1, 5 \\ 
\frac{227}{4}\,,  &\text{if } t \equiv 2, 4 \\ 
\frac{234}{4}\,,  &\text{if } t \equiv 3 
\end{cases} \qquad \mod{6} ;
\]
\[
a_1(t) = \begin{cases}
\frac{1968}{5}\,,  &\text{if } t \equiv 0 \\  
\frac{1158}{5}\,,  &\text{if } t \equiv 1,5 \\
\frac{1383}{5}\,,  &\text{if } t \equiv 2,10 \\
\frac{1653}{5}\,,  &\text{if } t \equiv 3 \\
\frac{1428}{5}\,,  &\text{if } t \equiv 4,8 \\
\frac{1923}{5}\,,  &\text{if } t \equiv 6 \\
\frac{1113}{5}\,,  &\text{if } t \equiv 7,11 \\
\frac{1698}{5}\,,  &\text{if } t \equiv 9 
\end{cases} \qquad \mod{12} ;
\]
and $a_0(t)$ is given in Table \ref{Tb:semiaffa0}.  

We call the constituents of the quasipolynomial 
\[
S_{\aa}(t) + 72 S_7(t) = \frac{1}{8} t^4 - \frac{9}{2} t^3 + a_2(t) t^2 - a_1(t) t + a_0(t)
\]
the \emph{truncated constituents} of $S_\aa(t)$, since they correspond to the truncated generating function mentioned just above.  
The $S_7$ term that undoes the truncation is 
\begin{align*}
S_7(t) &:= \left\lfloor\frac{t-1}{21}\right\rfloor + \begin{cases}
1, &\text{if } t \equiv 10, 13, 16, 17, 19, 20 \mod{21}; \\
0, &\text{otherwise}
\end{cases}	\\
&\,= \frac{t - \bar t}{21} + s_7(t) ,
\end{align*}
where $\bar t :=$ the least positive residue of $t$ modulo 7 and
$$
s_7(t) := \begin{cases} 
1, &\text{if } t \equiv 10, 13, 16, 17, 19, 20 \mod{21}; \\
 0,      &\text{otherwise}.
\end{cases}
$$
Note that $\bar{t} = 21$ if $t \equiv 0$, so that $S_7(0) = -1$ and in general $S_7(21k) = k-1$.

\begin{table}[htb]
\newcommand\Strut{\rule{0ex}{2.5ex}}
\begin{tabular}{|r|c||r|c||r|c||r|c||r|c||r|c|}

\hline
\Strut $t$	& $a_0(t)$	&$t$	& $a_0(t)$	&$t$	& $a_0(t)$	&$t$	& $a_0(t)$	&$t$	& $a_0(t)$	&$t$	& $a_0(t)$	\\[3pt]
\hline\hline
\Strut 0&$1224$		&20	&$524$		&40	&$584$		&60	&$1188$		&80	&$560$		&100	&$548$	\\[3pt]
\hline
\Strut 1&$\frac{7259}{40}$	&21&$\frac{31419}{40}$	&41&$\frac{6299}{40}$	&61&$\frac{8699}{40}$ 	&81&$\frac{29979}{40}$	&101&$\frac{7739}{40}$	\\[3pt]
\hline
\Strut 2&$\frac{1801}{5}$	&22&$\frac{1741}{5}$	&42&$\frac{5121}{5}$	&62&$\frac{1621}{5}$	&82&$\frac{1921}{5}$	&102&$\frac{4941}{5}$	\\[3pt]
\hline
\Strut 3&$\frac{23067}{40}$	&23&$\frac{827}{40}$	&43&$\frac{347}{40}$	&63&$\frac{24507}{40}$	&83&$\frac{-613}{40}$	&103&$\frac{1787}{40}$	\\[3pt]
\hline
\Strut 4&$\frac{2452}{5}$	&24&$\frac{5832}{5}$	&44&$\frac{2332}{5}$	&64&$\frac{2632}{5}$	&84&$\frac{5652}{5}$	&104&$\frac{2512}{5}$	\\[3pt]
\hline
\Strut 5&$\frac{2239}{8}$	&25&$\frac{2143}{8}$	&45&$\frac{6975}{8}$	&65&$\frac{1951}{8}$	&85&$\frac{2431}{8}$	&105&$\frac{6687}{8}$ 	\\[3pt]
\hline
\Strut 6&$\frac{4653}{5}$	&26&$\frac{1513}{5}$	&46&$\frac{1453}{5}$	&66&$\frac{4833}{5}$	&86&$\frac{1333}{5}$	&106&$\frac{1633}{5}$	\\[3pt]
\hline
\Strut 7&$\frac{5243}{40}$	&27&$\frac{26523}{40}$	&47&$\frac{4283}{40}$	&67&$\frac{3803}{40}$	&87&$\frac{27963}{40}$	&107&$\frac{2843}{40}$	\\[3pt]
\hline
\Strut 8&$\frac{2224}{5}$	&28&$\frac{2164}{5}$	&48&$\frac{5544}{5}$	&68&$\frac{2044}{5}$	&88&$\frac{2344}{5}$	&108&$\frac{5364}{5}$	\\[3pt]
\hline
\Strut 9&$\frac{31131}{40}$	&29&$\frac{8891}{40}$	&49&$\frac{8411}{40}$	&69&$\frac{32571}{40}$	&89&$\frac{7451}{40}$	&109&$\frac{9851}{40}$	\\[3pt]
\hline
\Strut 10&$413$		&30&$1017$		&50&$389$		&70&$377$		&90&$1053$		&110&$353$	\\[3pt]
\hline
\Strut 11&$\frac{539}{40}$	&31&$\frac{2939}{40}$	&51&$\frac{24219}{40}$	&71&$\frac{1979}{40}$	&91&$\frac{1499}{40}$	&111&$\frac{25659}{40}$ \\[3pt]
\hline
\Strut 12&$\frac{5796}{5}$	&32&$\frac{2656}{5}$	&52&$\frac{2596}{5}$	&72&$\frac{5976}{5}$	&92&$\frac{2476}{5}$	&112&$\frac{2776}{5}$	\\[3pt]
\hline
\Strut 13&$\frac{7547}{40}$	&33&$\frac{28827}{40}$	&53&$\frac{6587}{40}$	&73&$\frac{6107}{40}$	&93&$\frac{30267}{40}$	&113&$\frac{5147}{40}$ 	\\[3pt]
\hline
\Strut 14&$\frac{1477}{5}$	&34&$\frac{1777}{5}$	&54&$\frac{4797}{5}$	&74&$\frac{1657}{5}$	&94&$\frac{1597}{5}$	&114&$\frac{4977}{5}$	\\[3pt]
\hline
\Strut 15&$\frac{5823}{8}$	&35&$\frac{799}{8}$		&55&$\frac{1279}{8}$	&75&$\frac{5535}{8}$	&95&$\frac{1087}{8}$	&115&$\frac{991}{8}$	\\[3pt]
\hline
\Strut 16&$\frac{2488}{5}$	&36&$\frac{5508}{5}$	&56&$\frac{2368}{5}$	&76&$\frac{2308}{5}$	&96&$\frac{5688}{5}$	&116&$\frac{2188}{5}$	\\[3pt]
\hline
\Strut 17&$\frac{8603}{40}$	&37&$\frac{11003}{40}$	&57&$\frac{32283}{40}$	&77&$\frac{10043}{40}$	&97&$\frac{9563}{40}$	&117&$\frac{33723}{40}$ \\[3pt]
\hline
\Strut 18&$\frac{4689}{5}$	&38&$\frac{1189}{5}$	&58&$\frac{1489}{5}$	&78&$\frac{4509}{5}$	&98&$\frac{1369}{5}$	&118&$\frac{1309}{5}$	\\[3pt]
\hline
\Strut 19&$\frac{2651}{40}$	&39&$\frac{26811}{40}$	&59&$\frac{1691}{40}$	&79&$\frac{4091}{40}$	&99&$\frac{25371}{40}$	&119&$\frac{3131}{40}$	\\[3pt]
\hline
\end{tabular}
\vspace{15pt}
\caption{Constant terms of the truncated constituents of $S_\aa(t)$.}
\label{Tb:semiaffa0}
\end{table}

The period of the constant term of the truncated constituents is $120$.  It follows that $S_\aa(t)$ has period $7 \cdot 840$, that is, $840$.

The principal constituent of $S_\aa(t)$ (that is, for $t \equiv 0$) is
$$
\frac{1}{8} t^4 - \frac{9}{2} t^3 + \frac{243}{4} t^2 - \frac{13896}{35} t + 1296 .
$$
(This incorporates the effect of the term $-72S_7$.)  
The constant term is the same as in the cubic count, as it is the number of order types of semimagic squares.

We give the first few nonzero values of $S_\aa(t)$ in the following table.  (This sequence is A173547 in the OEIS \cite{OEIS}.)  The third row is the number of normalized squares, or symmetry classes (sequence A173725); this is $S_\aa(t)/72$.  The last rows are the numbers of reduced squares (sequence A173728) and of reduced, normalized squares (sequence A173726) with magic sum $t$.
\[
\begin{tabular}{r||c|c|c|c|c|c|c|c|c|c|c|c|c|c|c|c|c|c|}
$t$     	&12&13&14&15&16&17&18&19&20&21&22&23&24\\
\hline \hline
$S_\aa(t)$&0&0&0&72&144&288&576&864&1440&2088&3024&3888&5904\\
\hline
$\hS_\aa(t)$&0&0&0&1&2&4&8&12&20&29&42&54&82\\
\hline
$R_\aa(t)$&72&144&288&504&720&1152&1512&2160&2448&3816&3960&5544&6264\\
\hline
$r_\aa(t)$	&1&2&4&7&10&16&21&30&34&53&55&77&87
\end{tabular}
\]
%{Values from direct count of reduced, normalized squares ({\tt Maple}) and recursion for $\hS$, June 2004, and from LattE/Maple.  Reduced from LattE/Maple.}

%%%%%%
\subsubsection{Alternative methods:  Direct counting and direct computation}
\label{sa3alt}

We verified our formulas by computing $S_\aa(t)$ for $t \leq 100$ through direct enumeration of normal squares.  The results agree with those computed by expanding the generating function. 

We also applied Proposition \ref{T:snsemimagic} to derive a formula, independent of all other methods, by which we calculated numbers (which we are not describing; see the ``Six Little Squares'' Web page \cite{Maplefiles}) that allowed us to find the 840 constituents by interpolation.  These interpolated constituents fully agreed with the ones given above.

%%%%%%%%%%%%%%%%%%%%%%%%%%%%%%%%%%%%%%%%%%%%%%%%%%%%%%%%%%%%%%%%%%
\section{Magilatin squares of order 3} \label{magilatin}

A magilatin square is like a semimagic square except that entries may be equal if they are in different rows and columns.  The inside-out polytope is the same as with semimagic squares except that we omit those hyperplanes that prevent equality of entries in different rows and columns.  Thus, in our count of reduced squares, we have to count the fractional lattice points in some of the faces of the polytope.

The reduced normal form of a magilatin square is the same as that of a semimagic square except that the restrictions are weaker.  It might be thought that this would introduce ambiguity into the standard form because the minimum can occur in several cells, but it turns out that it does not.

\begin{prop}  \label{T:snmagiclatin}
A reduced, normal $3\times3$ magilatin square has the form \eqref{E:reducednormal} with the restrictions
\begin{equation}
\label{E:snmagilatinrange}
\begin{aligned}
 0 &< \beta, \gamma ;    \\
 0 &\leq \alpha ;    \\
 0 &\leq \delta \leq \beta ;
\end{aligned} 
\end{equation}
and \eqref{E:reducednormrestr}.
Each reduced square with $w$ in the upper right corner corresponds to exactly $t-w-1$ different magilatin squares with entries in the range $(0,t)$, for $0<w<t$. Each reduced square with magic sum $s$ corresponds to one magilatin square with magic sum equal to $t$, if $t \equiv s \mod{3}$, and none otherwise, for $0<s<t$.
\end{prop}

\begin{proof}
The proof is similar to that for semimagic squares; we can arrange the square by permuting rows and columns and by reflection in the main diagonal so that $x_{11}$ is the smallest entry, the first row and column are each increasing, and $x_{21} \geq x_{12}$.  We cannot say $x_{21} > x_{12}$ because entries that do not share a row or column may be equal.  Still, we obtain the form \eqref{E:reducednormal} with the bounds \eqref{E:snmagilatinrange} and the same inequations \eqref{E:reducednormrestr} as in semimagic because all the latter depend on having no two equal values in the same line (row or column).
\end{proof}

Each reduced, normal magilatin square gives rise to a family of true magilatin squares by adding a positive constant to each entry and by symmetries, which are generated by row and column permutations and reflection in the main diagonal.  Call the set of symmetries $\fG$.  
As with semimagic, $|\fG| = 2 (3!)^2 = 72$.  
Each normal, reduced square $S$ gives rise to $|\fG/\fF|=72/|\fF|$ squares via symmetries, where $\fF$ is the stabilizer subgroup of $S$.  If all entries are distinct, then the square is semimagic, $\fF$ is trivial, and everything is as with semimagic squares.  However, if $\alpha=0$ or $\delta=0$ or $\delta=\beta$, the stabilizer is nontrivial.  We consider each case in turn.

\emph{The case $\alpha=0<\delta$.}
Here $\delta<\beta$ because no line can repeat a value.  To fix the square we cannot permute any rows or columns but we can reflect in the main diagonal, so $|\fF|=2$.  Moreover, \eqref{E:reducednormrestr} reduces to
$$
\delta \neq \gamma, \frac{\beta}{2}, \frac{\beta+\gamma}{2} .
$$
We are in $OBC$, the $x=0$ facet of $Q$, with the induced arrangement of three lines, $\cI^{x=0}$.  
The number of reduced magilatin squares of this kind is $36r_{OBC}(t)$, where $r_{OBC}(t)$ is the number of $\frac1t$-lattice points in the open facet and, equivalently, the number of symmetry types of reduced magilatin squares of this kind.  We apply Equations \eqref{E:ehrrecipgf}--\eqref{E:ehrhypgf} to the intersection poset $\cL(OBC^\circ,\cI^{x=0})$, which is found in Figure \ref{F:ml3intersections}.  The M\"obius function $\mu(OBC,u)$ equals $(-1)^{\codim u}$.

\emph{The case $\delta=0<\alpha$.}
In this case a nontrivial member of $\fF$ can only exchange the two zero positions.  Such a symmetry that preserves the increase of the first row and column is unique (as one can easily see); thus $|\fF|=2$.  Furthermore, \eqref{E:reducednormrestr} reduces to
$$
\alpha \neq \beta .
$$
We are in $OAB$, the $z=0$ facet, with the induced arrangement $\cI^{z=0}$ of one line.  The number of reduced magilatin squares of this kind is $36r_{OAB}(t)$, where $r_{OAB}(t)$ is the number of $\frac1t$-lattice points in the open facet and, equivalently, the number of symmetry types of reduced squares.  We apply Equations \eqref{E:ehrrecipgf}--\eqref{E:ehrhypgf} to the intersection poset $\cL(OAB^\circ,\cI^{z=0})$, shown in Figure \ref{F:ml3intersections}.  The M\"obius function $\mu(OAB,u)$ equals $(-1)^{\codim u}$.

\emph{The case $\delta=\beta$.}
Here we must have $\alpha>0$.  There are two zero positions in opposite corners.  A symmetry that exchanges them and preserves increase in the first row and column is uniquely determined, so $|\fF|=2$.  The inequations reduce to 
$$
\beta \neq \gamma, \alpha+\gamma .
$$
We are in $OAC$, the facet where $y=z$, with the two-line induced arrangement $\cI^{y=z}$.  The number of reduced magilatin squares of this kind is $36r_{OAC}(t)$, where $r_{OAC}(t)$ is the number of $\frac1t$-lattice points in the open facet, equally the number of reduced symmetry types.  We apply Equations \eqref{E:ehrrecipgf}--\eqref{E:ehrhypgf} to the intersection poset $\cL(OAC^\circ,\cI^{y=z})$ in Figure \ref{F:ml3intersections}.  The M\"obius function $\mu(OAC,u)$ equals $(-1)^{\codim u}$.

\emph{The case $\alpha=0=\delta$.}
In these squares there are three zero positions and the whole square is a cyclic latin square.  Any symmetry that fixes the zero positions also fixes the rest of the square.  There are $3!$ symmetries that permute the zero positions, generated by row and column permutations.  They all preserve the entire square.  Therefore $|\fF|=6$.  The inequations disappear.  We are in the edge $OB$, which is the face where $x=z=0$, with the empty arrangement, $\cI^{x=z=0}=\eset$.  The number of reduced magilatin squares of this kind is $12r_{OB}(t)$, where $r_{OB}(t)$ is the number of $\frac1t$-lattice points in the open edge, also the number of reduced symmetry types.  The intersection poset $\cL(OB^\circ,\eset)$ consists of the one element $OB$, whose M\"obius function $\mu(OB,OB)=1$.

To get the intersection posets we may examine Tables \ref{Tb:planeedge} and \ref{Tb:linepolytope} to find the edges and vertices of $(Q^\circ,\cI)$ in each closed facet.  We also need to know which vertex is in which edge; this is easy.  Although we do not need the fourth facet, $ABC$, we include it for the interest of its more complicated geometry.
\begin{figure}
\begin{center}
\psfrag{LOAB}[c]{{\large$\cL(OAB^\circ,\cI^{z=0})$}}
\psfrag{LOAC}[c]{{\large$\cL(OAC^\circ,\cI^{y=z})$}}
\psfrag{LOBC}[c]{{\large$\cL(OBC^\circ,\cI^{x=0})$}}
\psfrag{LABC}[c]{{\large$\cL(AB C^\circ,\cI^{x+y=1/2})$}}
\psfrag{OAB}[l]{$OAB$}
\psfrag{OAC}[l]{$OAC$}
\psfrag{OBC}[l]{$OBC$}
\psfrag{ABC}[l]{$ABC$}
\psfrag{OE}[l]{$OE$}
\psfrag{AD}[l]{$AD$}
\psfrag{DE'}[l]{$DE'$}
\psfrag{OE''}[l]{$OE''$}
\psfrag{BD}[l]{$BD$}
\psfrag{DE''}[l]{$DE''$}
\psfrag{EE''}[l]{$EE''$}
\psfrag{BE'}[l]{$BE'$}
\psfrag{AE''}[l]{$AE''$}
\psfrag{CE}[l]{$CE$}
\psfrag{E'E''}[l]{$E'E''$}
\psfrag{F}[l]{$F$}
\psfrag{G}[l]{$G$}
\psfrag{G'}[l]{$G'$}
\psfrag{G''}[l]{$G''$}
\includegraphics{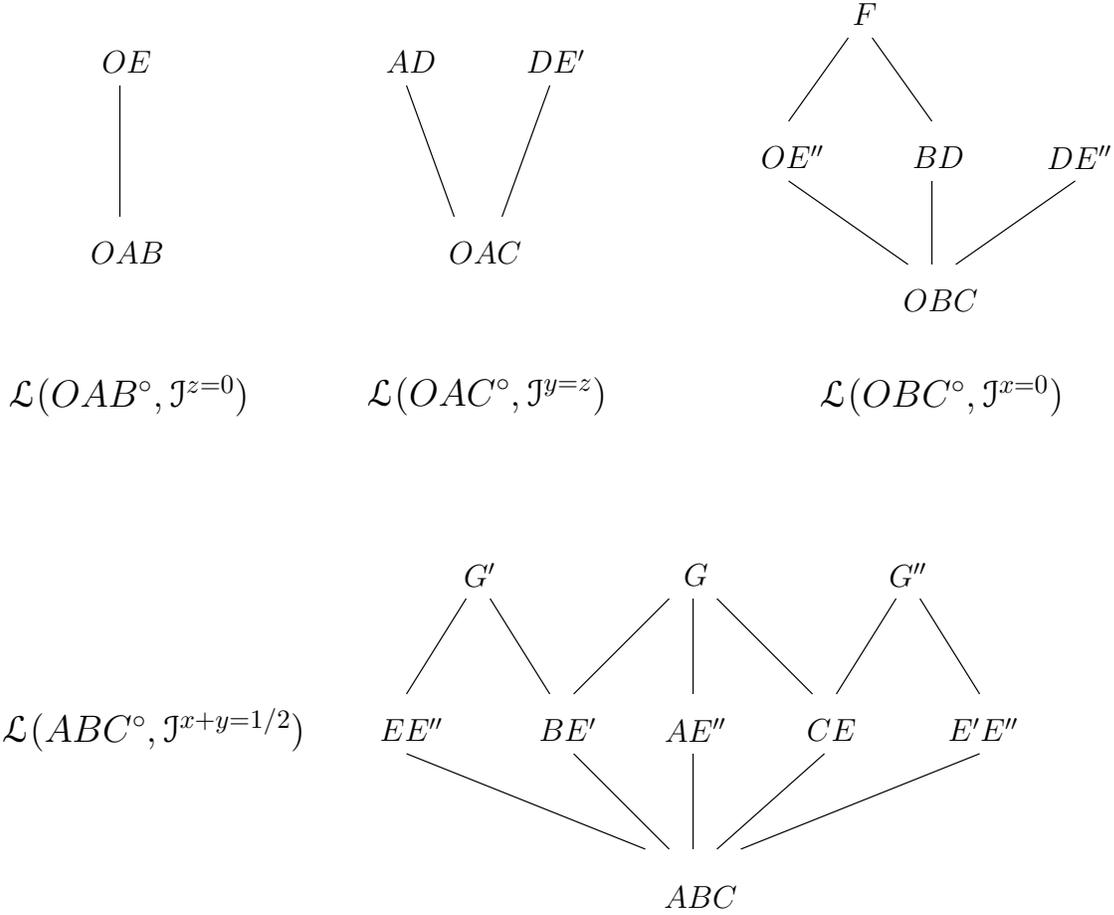}
\caption{The four facet intersection posets of $(Q^\circ,\cI)$.}
\label{F:ml3intersections}
\end{center}
\end{figure}

Of course all the functions $r_\sss$ and $R_\ml$ depend on whether we are counting cubically or affinely (thus, subscripted $\cc$ or $\aa$); the two types will be treated separately.  But the general conclusions hold that
\begin{equation} \label{E:ml3gfreduced}
\bR_\ml(x) 
= 72\br_\sss(x) + 36[ \br_{OAB}(x) + \br_{OAC}(x) + \br_{OBC}(x) ] + 12\br_{OB}(x) ,
\end{equation}
and for the number of reduced symmetry types, $\hR_\ml(t)$ with generating function $\hbR_\ml(x)$, 
\begin{equation} \label{E:ml3symgfreduced}
\hbR_\ml(x) 
= \br_\sss(x) + \br_{OAB}(x) + \br_{OAC}(x) + \br_{OBC}(x) + \br_{OB}(x) ,
\end{equation}
where $\br_\sss(x)$ is from semimagic and, by Equation \eqref{E:ehrhypgf} since $|\mu(X,Y)|=1$ for every lower interval in each facet poset (except facet $ABC$, which we do not use), 
\begin{equation} \label{E:ml3supernormfacets}
\begin{aligned}
(-1)^3 \br_{OAB}(1/x) &= \bE_{OAB}(x) + \bE_{OE}(x) , \\
(-1)^3 \br_{OAC}(1/x) &= \bE_{OAC}(x) + \bE_{AD}(x) + \bE_{DE'}(x) , \\
(-1)^3 \br_{OBC}(1/x) &= \bE_{OBC}(x) + \bE_{OE''}(x) + \bE_{BD}(x) + \bE_{DE''}(x) + \bE_{F}(x) , \\
(-1)^2 \br_{OB}(1/x) &= \bE_{OB}(x) ;
\end{aligned}
\end{equation}
the sign and reciprocal on the left result from Equation \eqref{E:gfmobius}.

There is also the generating function of the number of cubical or affine symmetry classes, $\hL(t)$, whose generating function is $\bhL(x)$.  This is obtained from $\hR_\ml(t)$ in the same way as $L(t)$ is from $R_\ml(t)$, the exact way depending on whether the count is affine or cubic.

%%%%%%%%%%%%%%%%
\subsection{Magilatin squares:  Cubical count (by upper bound)} \label{mlc3}

The weak quasipolynomial is exactly as in the semimagic cubical problem.

%%%%%%
\subsubsection{Magilatin squares by upper bound}
\label{mlc3results}

The number of $3\times3$ magilatin squares with strict upper bound $t$ is $L_\cc(t)$.  We count them via $R_{\mlc}(w)$, the number of reduced magilatin squares with largest entry (which we know to be $x_{13}$) equal to $w$.  The formula is
\begin{equation} 
\label{E:mlc3count}
L_\cc(t) = \sum_{w=0}^{t-1} (t-1-w) R_{\mlc}(w) .
\end{equation}
Equivalently, $R_{\mlc}(w)$ counts $\frac{1}{w}$-integral points in the interior and part of the boundary of the inside-out polytope $(Q_\cc,\cI_\cc)$ of Section \ref{sc3}, weighted variably by $72/|\fF|$.

Now we must calculate the closed Ehrhart generating function for each necessary face.  This is done by {\tt LattE}; here are the results.  
First, $OAB_\cc$:
\begin{equation}
\label{E:mlc3sngfoab}
\begin{aligned}
(-1)^3 \br_{OAB_\cc}(1/x) =\ &\bE_{OAB_\cc}(x) + \bE_{OE_\cc}(x) \\
=\ &\frac{1}{ (1-x)^2 (1-x^2) } + \frac{1}{(1-x) (1-x)^3 }	\\
=\ & \frac{ x+2 }{ (1-x) (1-x^2) (1-x^3) } \ .
\end{aligned}
\end{equation}
Next is $OAC_\cc$:
\begin{equation}
\label{E:mlc3sngfoac}
\begin{aligned}
(-1)^3 \br_{OAC_\cc}(1/x) =\ &\bE_{OAC_\cc}(x) + \bE_{AD_\cc}(x) 
+ \bE_{D_\cc E_\cc'}(x) \\
=\ &\frac{1}{(1-x)^2 (1-x^2) } + \frac{1}{(1-x^2)^2 } + \frac{1}{(1-x^2) (1-x^3) }	\\
=\ & \frac{ x^2 + 2x + 3 }{ (1-x^2)^2 (1-x^3)  } \ .
\end{aligned}
\end{equation}
The last facet is $OB_\cc C_\cc$:
\begin{equation}
\label{E:mlc3sngfobc}
\begin{aligned}
(-1)^3 \br_{OB_\cc C_\cc}(1/x) =\ &\bE_{OB_\cc C_\cc}(x) + \bE_{OE_\cc''}(x) + \bE_{B_\cc D_\cc}(x) + \bE_{D_\cc E_\cc''}(x) + \bE_{F_\cc}(x) \\[4pt]
=\ &\frac{1}{(1-x)^3 } + \frac{1}{(1-x) (1-x^2) } + \frac{1}{(1-x) (1-x^2) }	\\[4pt]
&+ \frac{2x^2+1}{(1-x^2)^2 } + \frac{1}{1-x^3}	\\[4pt]
=\ & \frac{ -2x^5 + 4x^2 + 5x + 5 }{ (1-x^2)^2 (1-x^3)  } \ .
\end{aligned}
\end{equation}
Finally, the edge $OB_\cc$:
\begin{equation}
\label{E:mlc3sngfob}
(-1)^2 \br_{OB_\cc}(1/x) = \bE_{OB_\cc}(x) = \frac{1}{(1-x)^2 } \ .
\end{equation}

Now $\bR_{\mlc}(x)$ results from \eqref{E:ml3gfreduced}, and then from \eqref{E:mlc3count} we see that
\begin{equation}\label{E:mlc3gf}
\begin{aligned} 
&\quad \bL_\cc(x) = \frac{x^2}{(1-x)^2} \, \bR_{\mlc}(x) = \\[4pt]
&\frac{ 
 12 x^4 \left\{ \begin{aligned} 
 &79 x^{15} + 190 x^{14} + 260 x^{13} + 250 x^{12} + 211 x^{11} + 179 x^{10} + 181 x^9 \\
 &+ 198 x^8 + 210 x^7 + 181 x^6 + 125 x^5 + 61 x^4 + 22 x^3 + 8 x^2 + 4 x + 1 
 \end{aligned} \right\}
}{ 
(1-x^4) (1-x^5) (1-x^3)^2 (1-x^2)^2
} \ .
\end{aligned} 
\end{equation}

The constituents of $L_\cc$ are extracted as described in Section \ref{genmethod}, and here they are:
\begin{equation}\label{E:mlc-constituents}
L_\cc(t) = \begin{cases}
 \ds \frac{3}{10} t^5 - \frac{51}{8} t^4  + \frac{202}{3} t^3  -\frac{3769}{10} t^2 + c_1(t) t - c_0(t),	&\text{if $t$ is even}; \\[14pt]
 \ds \frac{3}{10} t^5 - \frac{51}{8} t^4 + \frac{202}{3} t^3  -\frac{7493}{20} t^2 + c_1(t) t - c_0(t),	&\text{if $t$ is odd};
\end{cases}
\end{equation}
where $c_1$ varies with period 6, given by 
$$
c_1(t) = \begin{cases}
994,		&\text{ if } t \equiv 0,2 \\
986,		&\text{ if } t \equiv 4  \\
\frac{1909}{2}\,, 	&\text{ if } t \equiv 1  \\
\frac{1925}{2}\,, 	&\text{ if } t \equiv 3, 5 \\
\end{cases} \qquad \mod{6};
$$
and $c_0$, given by Table \ref{Tb:magilatincubc0}, varies with period 60. 
\begin{table}[tbh]
\newcommand\Strut{\rule{0ex}{2.5ex}}
\begin{tabular}{|r|c||r|c||r|c||r|c||r|c|}

\hline
\Strut $t$	& $c_0(t)$	&$t$	& $c_0(t)$	&$t$	& $c_0(t)$	&$t$	& $c_0(t)$	&$t$	& $c_0(t)$	\\[3pt]
\hline\hline
\Strut 0&948			&12&948			&24&$\frac{4884}{5}$	&36&$\frac{4452}{5}$	&48&$\frac{4884}{5}$	\\[3pt]
\hline
\Strut 1&$\frac{76933}{120}$	&13&$\frac{87301}{120}$	&25&$\frac{16769}{24}$	&37&$\frac{16769}{24}$	&49&$\frac{87301}{120}$ \\[3pt]
\hline
\Strut 2&$\frac{2780}{3}$	&14&$\frac{14332}{15}$	&26&$\frac{13036}{15}$	&38&$\frac{14332}{15}$	&50&$\frac{2780}{3}$	\\[3pt]
\hline
\Strut 3&$\frac{35607}{40}$	&15&$\frac{6891}{8}$	&27&$\frac{6891}{8}$	&39&$\frac{35607}{40}$	&51&$\frac{32151}{40}$	\\[3pt]
\hline
\Strut 4&$\frac{13292}{15}$	&16&$\frac{11996}{15}$	&28&$\frac{13292}{15}$	&40&$\frac{2572}{3}$	&52&$\frac{2572}{3}$	\\[3pt]
\hline
\Strut 5&$\frac{18433}{24}$	&17&$\frac{18433}{24}$	&29&$\frac{95621}{120}$	&41&$\frac{85253}{120}$	&53&$\frac{95621}{120}$ \\[3pt]
\hline
\Strut 6&$\frac{4452}{5}$	&18&$\frac{4884}{5}$	&30&948			&42&948			&54&$\frac{4884}{5}$	\\[3pt]
\hline
\Strut 7&$\frac{18497}{24}$	&19&$\frac{95941}{120}$	&31&$\frac{85573}{120}$	&43&$\frac{95941}{120}$	&55&$\frac{18497}{24}$	\\[3pt]
\hline
\Strut 8&$\frac{14332}{15}$	&20&$\frac{2780}{3}$	&32&$\frac{2780}{3}$	&44&$\frac{14332}{15}$	&56&$\frac{13036}{15}$	\\[3pt]
\hline
\Strut 9&$\frac{32727}{40}$	&21&$\frac{29271}{40}$	&33&$\frac{32727}{40}$	&45&$\frac{6315}{8}$	&57&$\frac{6315}{8}$	\\[3pt]
\hline
\Strut 10&$\frac{2572}{3}$	&22&$\frac{2572}{3}$	&34&$\frac{13292}{15}$	&46&$\frac{11996}{15}$	&58&$\frac{13292}{15}$	\\[3pt]
\hline
\Strut 11&$\frac{93893}{120}$	&23&$\frac{104261}{120}$&35&$\frac{20161}{24}$	&47&$\frac{20161}{24}$	&59&$\frac{104261}{120}$ \\[3pt]
\hline
\end{tabular}
\vspace{15pt}
\caption{Constant terms of the constituents of $L_\cc(t)$, counting all magilatin squares by upper bound.}
\label{Tb:magilatincubc0}
\end{table}
Thus the period of $L_\cc$ turns out to be 60, just like that of $S_{\cc}$ (not a surprise). 
However, again as with $S_{\cc}$, the even constant terms have half the period of the odd constant terms.  That means $L_\cc(2t)$ has period equal to half the denominator of the corresponding inside-out polytope $(2P,\cH)$.  We have no explanation for this.

The principal constituent, that for $t\equiv0 \mod{60}$, is 
$$
\frac{3}{10} t^5 - \frac{51}{8} t^4  + \frac{202}{3} t^3  -\frac{3769}{10} t^2 + 994 t - 948 .
$$
The constant term for magilatin squares does not have the simple interpretation as a number of linear orderings that it does for magic and semimagic squares, because the entries in the square need not all be different.  (Still, there is an interpretation as a number of partial orderings of the nine cells; see Theorem 4.1 in our general magic and magilatin paper \cite{MML}, and recall that an acyclic orientation of a graph can be represented by a partial ordering of the vertices.)   

We confirmed the formulas by generating all magilatin squares and comparing the count with the coefficients of $\bL_\cc(x)$ up to $t = 91$.

%%%%%%
\subsubsection{Symmetry types of magilatin squares, counted by upper bound}
\label{mlc3symresults}

The number of symmetry types with strict bound $t$ is $\hL_\cc(t)$.  We count them via $\hR_{\mlc}(w)$, given by Equation \eqref{E:ml3symgfreduced}; then 
\begin{equation} 
\label{E:mlc3symcount}
\hL_\cc(t) = \sum_{w=0}^{t-1} (t-1-w) \, \hR_{\mlc}(w) .
\end{equation}
Equivalently, $\hR_{\mlc}(w)$ counts the $\frac{1}{w}$-integral points in the interior and part of the boundary of the inside-out polytope $(Q_\cc,\cI_\cc)$ of Section \ref{sc3}.

We get $\hbR_{\mlc}(x)$ from \eqref{E:ml3symgfreduced}; then from \eqref{E:mlc3symcount} we see that
\begin{equation}\label{E:mlc3gfsym}
\begin{aligned} 
&\hbL_\cc(x) = \frac{x^2}{(1-x)^2} \, \hbR_{\mlc}(x) = \\[4pt]
&\quad \frac{x^4 \left\{ \begin{aligned}
9x^{15}&+20x^{14}+23x^{13}+16x^{12}+10x^{11}+13x^{10}+27x^{9}+43x^{8}\\
&+54x^{7}+52x^{6}+41x^{5}+25x^{4}+14x^{3}+8x^{2}+4x+1
\end{aligned} \right\}
}{ 
(1-x^2)^2 (1-x^3)^2 (1-x^4) (1-x^5)
} \ .
\end{aligned}
\end{equation}

The constituents of $\hL_\cc$ are:
$$
\hL_\cc(t) = \begin{cases}
\ds \frac{1}{240} t^5 - \frac{3}{64} t^4  + \frac{97}{216} t^3  -\frac{2029}{720} t^2 + \hat c_1(t) t - \hat c_0(t),	&\text{if $t$ is even}; \\ \\
\ds \frac{1}{240} t^5 - \frac{3}{64} t^4  + \frac{97}{216} t^3  -\frac{4013}{1440} t^2 + \hat c_1(t) t - \hat c_0(t),	&\text{if $t$ is odd};
\end{cases}
$$
where $\hat c_1$ varies with period 6, given by 
$$
\hat c_1(t) = \begin{cases}
\frac{17}{2}\,,		&\text{ if } t \equiv 0, 2 \\
\frac{151}{18}\,,		&\text{ if } t \equiv 4  \\
\frac{1163}{144}\,,	&\text{ if } t \equiv 1  \\
\frac{131}{16}\,,		&\text{ if } t \equiv 3, 5 \\
\end{cases} \qquad \mod{6},
$$
and $\hat c_0$, given by Table \ref{Tb:magilatincubsymc0}, varies with period 60.
\begin{table}[htb]
\newcommand\Strut{\rule{0ex}{2.5ex}}
\begin{tabular}{|r|c||r|c||r|c||r|c||r|c|}

\hline
\Strut $t$	& $\hat c_0(t)$	&$t$	& $\hat c_0(t)$	&$t$	& $\hat c_0(t)$	&$t$	& $\hat c_0(t)$	&$t$	& $\hat c_0(t)$	\\[3pt]
\hline\hline
\Strut 0&9			&12&9			&24&$\frac{47}{5}$	&36&$\frac{41}{5}$	&48&$\frac{47}{5}$	\\[3pt]
\hline
\Strut 1&$\frac{49213}{8640}$	&13&$\frac{59581}{8640}$&25&$\frac{11225}{1728}$&37&$\frac{11225}{1728}$&49&$\frac{59581}{8640}$ \\[3pt]
\hline
\Strut 2&$\frac{235}{27}$	&14&$\frac{1229}{135}$	&26&$\frac{1067}{135}$	&38&$\frac{1229}{135}$	&50&$\frac{235}{27}$	\\[3pt]
\hline
\Strut 3&$\frac{2823}{320}$	&15&$\frac{539}{64}$	&27&$\frac{539}{64}$	&39&$\frac{2823}{320}$	&51&$\frac{2439}{320}$	\\[3pt]
\hline
\Strut 4&$\frac{1144}{135}$	&16&$\frac{982}{135}$	&28&$\frac{1144}{135}$	&40&$\frac{218}{27}$	&52&$\frac{218}{27}$	\\[3pt]
\hline
\Strut 5&$\frac{12313}{1728}$	&17&$\frac{12313}{1728}$&29&$\frac{65021}{8640}$&41&$\frac{54653}{8640}$&53&$\frac{65021}{8640}$ \\[3pt]
\hline
\Strut 6&$\frac{41}{5}$		&18&$\frac{47}{5}$	&30&9			&42&9			&54&$\frac{47}{5}$	\\[3pt]
\hline
\Strut 7&$\frac{12953}{1728}$	&19&$\frac{68221}{8640}$&31&$\frac{57853}{8640}$&43&$\frac{68221}{8640}$&55&$\frac{12953}{1728}$ \\[3pt]
\hline
\Strut 8&$\frac{1229}{135}$	&20&$\frac{235}{27}$	&32&$\frac{235}{27}$	&44&$\frac{1229}{135}$	&56&$\frac{1067}{135}$	\\[3pt]
\hline
\Strut 9&$\frac{2503}{320}$	&21&$\frac{2119}{320}$	&33&$\frac{2503}{320}$	&45&$\frac{475}{64}$	&57&$\frac{475}{64}$	\\[3pt]
\hline
\Strut 10&$\frac{218}{27}$	&22&$\frac{218}{27}$	&34&$\frac{1144}{135}$	&46&$\frac{982}{135}$	&58&$\frac{1144}{135}$	\\[3pt]
\hline
\Strut 11&$\frac{63293}{8640}$	&23&$\frac{73661}{8640}$&35&$\frac{14041}{1728}$&47&$\frac{14041}{1728}$&59&$\frac{73661}{8640}$ \\[3pt]
\hline
\end{tabular}
\vspace{15pt}
\caption{Constant terms of the constituents of $\hL_\cc(t)$, counting symmetry types of magilatin squares by upper bound.}
\label{Tb:magilatincubsymc0}
\end{table}
Thus the period of $\hL_\cc$ turns out to be 60.  As with $S_{\cc}$ and $L_\cc$, the period of the even constant terms is half that of the odd constant terms.

The principal constituent of $\hL_\cc$ is 
$$
\frac{1}{240} t^5 - \frac{3}{64} t^4  + \frac{97}{216} t^3  -\frac{2029}{720} t^2 + \frac{17}{2} t - 9 .
$$

%%%%%%
\subsubsection{Some real numbers}
\label{mlc3numbers}

For the first several nonzero values of the numbers of magilatin squares and of symmetry types, consult this table:
\[
\begin{tabular}{r||c|c|c|c|c|c|c|c|c|c|c|c|c|c|c|c|c|}
$t$     	&4&5&6&7&8&9&10&11&12&13&14&15\\
\hline \hline
$L_\cc(t)$&12&48&120&384&1068&2472&4896&9072&15516&25608&40296&61608\\
\hline
$\hL_\cc(t)$    &1&4&10&24&53&106&191&328&528&822&1230&1794\\
\hline
$R_{\mlc}(t)$    &12&24&36&192&420&720&1020&1752&2268&3648&4596&6624\\
\hline
$r_{\mlc}(t)$    &1&2&3&8&15&24&32&52&63&94&114&156\\
\end{tabular}
\]
The third line contains the number of symmetry classes of $3\times3$ magilatin squares, counted by upper bound.  The main numbers, $L_\cc(t)$ and $\hL_\cc(t)$, are sequences A173548 and A173729 in the OEIS \cite{OEIS}.  The reduced numbers, $R_\mlc(t)$ and $\hR_\mlc(t)$, are sequences A174018 and A174019.  In contrast to the semimagic case, the number of squares is not a simple multiple of the number of symmetry types.

%%%%%%%%%%%%%%%%
\subsection{Magilatin squares:  Affine count (by magic sum)} \label{mla3}

The last example is $3\times3$ magilatin squares, counted affinely.  Let $L_\aa(t)$ be the number of $3\times3$ magilatin squares with magic sum $t>0$.

The weak quasipolynomial is the same as in affine semimagic.

%%%%%%
\subsubsection{Magilatin squares by magic sum}
 \label{mla3results}

We compute $L_\aa(t)$, the number of squares with magic sum $t$, via $R_{\mla}(s)$, the number of reduced squares with magic sum $s$.  The formula is
\begin{align} 
\label{E:mla3count}
L_\aa(t) &= 
\sum_{\substack{0 < s \leq t-3 \\ s \equiv t \mod{3}}} R_{\mla}(s) &\text{ if } t>0 .
\end{align}
Equivalently, $R_{\mla}(s)$ counts $\frac{1}{s}$-integral points in the interior and part of the boundary of the inside-out polytope $(Q_\aa,\cI_\aa)$ of Section \ref{sa3}, each weighted by $72/|\fF|$.

Now we calculate (by {\tt LattE}) the closed Ehrhart generating function for each necessary face.  First, $OAB_\aa$:
\begin{equation}
\label{E:mla3sngfoab}
\begin{aligned}
(-1)^3 \br_{OAB_\aa}(1/x) =\ &\bE_{OAB_\aa}(x) + \bE_{OE_\aa}(x) \\
=\ &\frac{1}{(1-x) (1-x^2)^2 } + \frac{1}{(1-x) (1-x^4) }	\\
=\ & \frac{ 2 }{ (1-x) (1-x^2) (1-x^4) } \ .
\end{aligned}
\end{equation}
Next is $OAC_\aa$:
\begin{equation}
\label{E:mla3sngfoac}
\begin{aligned}
(-1)^3 \br_{OAC_\aa}(1/x) =\ &\bE_{OAC_\aa}(x) + \bE_{AD_\aa}(x) 
+ \bE_{D_\aa E_\aa'}(x) \\
=\ &\frac{1}{(1-x) (1-x^2)^2 } + \frac{1}{(1-x^2) (1-x^3) } + \frac{1}{(1-x^3) (1-x^4) }	\\
=\ & \frac{ x^3 + x^2 + x + 3 }{ (1-x^2) (1-x^3) (1-x^4) } \ .
\end{aligned}
\end{equation}
The last facet is $OB_\aa C_\aa$:
\begin{equation}
\label{E:mla3sngfobc}
\begin{aligned}
(-1)^3 \br_{OB_\aa C_\aa}(1/x) =\ &\bE_{OB_\aa C_\aa}(x) + \bE_{OE_\aa''}(x) + \bE_{B_\aa D_\aa}(x) \\
&+ \bE_{D_\aa E_\aa''}(x) + \bE_{F_\aa}(x) \\
=\ &\frac{1}{ (1-x) (1-x^2)^2 } + \frac{1}{ (1-x) (1-x^4) } + \frac{1}{ (1-x^2) (1-x^3) }	\\
&+ \frac{1}{ (1-x^3) (1-x^4) } + \frac{1}{1-x^5}	\\[2pt]
=\ &\frac{x^5+4x^4+6x^3+7x^2+7x+5}{(1-x^5) (1-x^3) (1-x^4) (1+x) } \ .
\end{aligned}
\end{equation}
Finally, the edge $OB_\aa$:
\begin{equation}
\label{E:mla3sngfob}
\begin{aligned}
(-1)^2 \br_{OB_\aa}(1/x) = \bE_{OB_\aa}(x) 
= \frac{1}{(1-x) (1-x^2) } \ .
\end{aligned}
\end{equation}

Now we get $\bR_{\mla}(x)$ from \eqref{E:ml3gfreduced}; then by \eqref{E:mla3count} we deduce that
\begin{equation}\label{E:mla3gf}
\begin{aligned}
&\bL_\aa(x) = \frac{x^3}{1-x^3} \, \bR_{\mla}(x) \\[4pt]
&= \frac{
12x^6 \left\{ \begin{aligned}
&1+3x+7x^2+15x^3+33x^4+65x^5+128x^6+208x^7+316x^8\\
&+434x^9+566x^{10}+676x^{11}+784x^{12}+852x^{13}+911x^{14}+936x^{15}\\
&+967x^{16}+967x^{17}+1001x^{18}+995x^{19}+1000x^{20}+955x^{21}\\
&+893x^{22}+752x^{23}+624x^{24}+456x^{25}+322x^{26}+174x^{27}+79x^{28}
\end{aligned} \right\}
}{
(1+x) (1+x+x^2) (1+x^2) (1-x^3) (1-x^5) (1-x^6) (1-x^7) (1-x^8) 
} \ .
\end{aligned}
\end{equation}

The constituents of $L_\aa$ are the following:  
\begin{equation}\label{E:mla-constituents}
L_\aa(t) = \frac{1}{8} t^4 - 3 t^3  + a_2(t) t^2 - a_1(t) t + a_0(t)	- 72 S_7(t) ,
\end{equation}
where $a_2$ varies with period 6, given by 
$$
a_2(t) = \begin{cases}
 \frac{151}{4}\,,	&\text{ if } t \equiv 0 \\
 \frac{135}{4}\,,	&\text{ if } t \equiv 2, 4  \\
 \frac{63}{2}\,,	&\text{ if } t \equiv 1, 5  \\
 \frac{71}{2}\,,	&\text{ if } t \equiv 3 \\
\end{cases} \qquad \mod{6};
$$
the linear coefficient varies with period 12, given by 
$$
a_1(t) = \begin{cases}
 \frac{1296}{5}\,,		&\text{ if } t \equiv 0 \\
 \frac{1347}{10}\,,	&\text{ if } t \equiv 1,5 \\
 \frac{831}{5}\,,		&\text{ if } t \equiv 2,10 \\
 \frac{2097}{10}\,,	&\text{ if } t \equiv 3 \\
 \frac{876}{5}\,,		&\text{ if } t \equiv 4,8 \\
 \frac{1251}{5}\,,		&\text{ if } t \equiv 6 \\
 \frac{1257}{10}\,,	&\text{ if } t \equiv 7,11 \\
 \frac{2187}{10}\,,	&\text{ if } t \equiv 9 \\
\end{cases} \qquad \mod{12};
$$
the constant term $a_0$, given by Table \ref{Tb:magilatinaffc0}, varies with period 120; and $S_7$ is as in the affine semimagic count. 

\begin{table}[htb]
\newcommand\Strut{\rule{0ex}{2.5ex}}
\begin{tabular}{|r|c||r|c||r|c||r|c||r|c||r|c|}
\hline
\Strut $t$	& $a_0(t)$	&$t$	& $a_0(t)$	&$t$	& $a_0(t)$	&$t$	& $a_0(t)$	&$t$	& $a_0(t)$	&$t$	& $a_0(t)$	\\[3pt]
\hline\hline
\Strut 0&876			&20&340			&40&400			&60&840			&80&376			&100&364		\\[3pt]
\hline
\Strut 1&$\frac{4243}{40}$	&21&$\frac{21843}{40}$	&41&$\frac{3283}{40}$	&61&$\frac{5683}{40}$ 	&81&$\frac{20403}{40}$	&101&$\frac{4723}{40}$	\\[3pt]
\hline
\Strut 2&$\frac{1097}{5}$	&22&$\frac{1037}{5}$	&42&$\frac{3597}{5}$	&62&$\frac{917}{5}$	&82&$\frac{1217}{5}$	&102&$\frac{3417}{5}$	\\[3pt]
\hline
\Strut 3&$\frac{15219}{40}$	&23&$\frac{-461}{40}$	&43&$\frac{-941}{40}$	&63&$\frac{16659}{40}$	&83&$\frac{-1901}{40}$	&103&$\frac{499}{40}$	\\[3pt]
\hline
\Strut 4&$\frac{1604}{5}$	&24&$\frac{4164}{5}$	&44&$\frac{1484}{5}$	&64&$\frac{1784}{5}$	&84&$\frac{3984}{5}$	&104&$\frac{1664}{5}$	\\[3pt]
\hline
\Strut 5&$\frac{1463}{8}$	&25&$\frac{1367}{8}$	&45&$\frac{4887}{8}$	&65&$\frac{1175}{8}$	&85&$\frac{1655}{8}$	&105&$\frac{4599}{8}$ 	\\[3pt]
\hline
\Strut 6&$\frac{3201}{5}$	&26&$\frac{881}{5}$	&46&$\frac{821}{5}$	&66&$\frac{3381}{5}$	&86&$\frac{701}{5}$	&106&$\frac{1001}{5}$	\\[3pt]
\hline
\Strut 7&$\frac{3091}{40}$	&27&$\frac{17811}{40}$	&47&$\frac{2131}{40}$	&67&$\frac{1651}{40}$	&87&$\frac{19251}{40}$	&107&$\frac{691}{40}$	\\[3pt]
\hline
\Strut 8&$\frac{1448}{5}$	&28&$\frac{1388}{5}$	&48&$\frac{3948}{5}$	&68&$\frac{1268}{5}$	&88&$\frac{1568}{5}$	&108&$\frac{3768}{5}$	\\[3pt]
\hline
\Strut 9&$\frac{21267}{40}$	&29&$\frac{5587}{40}$	&49&$\frac{5107}{40}$	&69&$\frac{22707}{40}$	&89&$\frac{4147}{40}$	&109&$\frac{6547}{40}$	\\[3pt]
\hline
\Strut 10&265			&30&705			&50&241			&70&229			&90&741			&110&205		\\[3pt]
\hline
\Strut 11&$\frac{-1037}{40}$	&31&$\frac{1363}{40}$	&51&$\frac{16083}{40}$	&71&$\frac{403}{40}$	&91&$\frac{-77}{40}$	&111&$\frac{17523}{40}$	\\[3pt]
\hline
\Strut 12&$\frac{4092}{5}$	&32&$\frac{1772}{5}$	&52&$\frac{1712}{5}$	&72&$\frac{4272}{5}$	&92&$\frac{1592}{5}$	&112&$\frac{1892}{5}$	\\[3pt]
\hline
\Strut 13&$\frac{4819}{40}$	&33&$\frac{19539}{40}$	&53&$\frac{3859}{40}$	&73&$\frac{3379}{40}$	&93&$\frac{20979}{40}$	&113&$\frac{2419}{40}$ 	\\[3pt]
\hline
\Strut 14&$\frac{809}{5}$	&34&$\frac{1109}{5}$	&54&$\frac{3309}{5}$	&74&$\frac{989}{5}$	&94&$\frac{929}{5}$	&114&$\frac{3489}{5}$	\\[3pt]
\hline
\Strut 15&$\frac{4023}{8}$	&35&$\frac{311}{8}$	&55&$\frac{791}{8}$	&75&$\frac{3735}{8}$	&95&$\frac{599}{8}$	&115&$\frac{503}{8}$	\\[3pt]
\hline
\Strut 16&$\frac{1676}{5}$	&36&$\frac{3876}{5}$	&56&$\frac{1556}{5}$	&76&$\frac{1496}{5}$	&96&$\frac{4056}{5}$	&116&$\frac{1376}{5}$	\\[3pt]
\hline
\Strut 17&$\frac{5011}{40}$	&37&$\frac{7411}{40}$	&57&$\frac{22131}{40}$	&77&$\frac{6451}{40}$	&97&$\frac{5971}{40}$	&117&$\frac{23571}{40}$	\\[3pt]
\hline
\Strut 18&$\frac{3273}{5}$	&38&$\frac{593}{5}$	&58&$\frac{893}{5}$	&78&$\frac{3093}{5}$	&98&$\frac{773}{5}$	&118&$\frac{713}{5}$	\\[3pt]
\hline
\Strut 19&$\frac{787}{40}$	&39&$\frac{18387}{40}$	&59&$\frac{-173}{40}$	&79&$\frac{2227}{40}$	&99&$\frac{16947}{40}$	&119&$\frac{1267}{40}$	\\[3pt]
\hline
\end{tabular}
\vspace{15pt}
\caption{Constant terms of the truncated constituents of $L_\aa(t)$, counting magilatin squares by magic sum.}
\label{Tb:magilatinaffc0}
\end{table}

The period of $L_\aa$ turns out to be 840---the period of the constant terms, due to the combination of $a_0(t)$ and the constant term of $S_7(t)$.  This is equal to the denominator.  

The principal constituent of $L_\aa$, that is, for $t\equiv0 \mod{840}$, is 
$$
\frac{1}{8} t^4 - 3 t^3 + \frac{151}{4} t^2 - \frac{9192}{35} t + 948 
$$
(incorporating the $-72S_7$ term).  
(As with the cubical magilatin count, there is an interpretation of the constant term 948 in terms of partial orderings; see Theorem 4.7 in our general paper \cite{MML}.)  

We verified the results by comparing an actual count of magilatin squares with magic sum $t \leq 100$ to the coefficients in $\bL_\aa$.

%%%%%%
\subsubsection{Symmetry types of magilatin squares, counted by magic sum}
 \label{mla3symresults}

We compute $\hL_\aa(t)$, the number of symmetry types of squares with magic sum $t$, via $\hR_{\mla}(s)$, the number of reduced symmetry types with magic sum $s$.  The formula is
\begin{align} 
\label{E:mla3symcount}
\hL_\aa(t) &= \sum_{\substack{0 < s \leq t-3 \\ s \equiv t \mod{3}}} \hR_{\mla}(s) &\text{ if } t>0 .
\end{align}
Equivalently, $\hR_{\mla}(s)$ counts $\frac{1}{s}$-integral points in the interior and part of the boundary of the inside-out polytope$(Q_\aa,\cI_\aa)$ of Section \ref{sa3}.

 From \eqref{E:ml3gfreduced} we get $\hbR_{\mla}(x)$ and then from \eqref{E:mla3symcount} we see that
\begin{equation}\label{E:mla3symgf}
\begin{aligned} 
&\hbL_\aa(x) = \frac{x^3}{1-x^3} \, \hbR_{\mla}(x) \\[4pt]
&= \frac{
x^6 \left\{ \begin{aligned} 
&1+3x+7x^2+13x^3+23x^4+37x^5+60x^6+86x^7+118x^8+149x^9\\
&+180x^{10}+199x^{11}+212x^{12}+208x^{13}+196x^{14}+171x^{15}\\
&+145x^{16}+115x^{17}+96x^{18}+79x^{19}+72x^{20}+67x^{21}\\
&+66x^{22}+59x^{23}+54x^{24}+43x^{25}+33x^{26}+19x^{27}+9x^{28}
\end{aligned} \right\}
}{
(1+x) (1+x+x^2) (1+x^2) (1-x^3) (1-x^5) (1-x^6) (1-x^7) (1-x^8) 
} \ .
\end{aligned}
\end{equation}

The constituents of $\hL_\aa$ are:                 
 $$
 \hL_\aa(t) = \frac{1}{576} t^4 - \frac{1}{48} t^3 + \hat a_2(t) t^2 - \hat a_1(t) t + \hat a_0(t) - S_7(t) , 
 $$
 where the quadratic term varies with period 6, given by 
 $$
 \hat a_2(t) = \begin{cases}
 \frac{25}{96}\,,		&\text{ if } t \equiv 0 \\
 \frac{25}{144}\,,		&\text{ if } t \equiv 1,5 \\
 \frac{59}{288}\,,		&\text{ if } t \equiv 2,4 \\
 \frac{11}{48}\,,		&\text{ if } t \equiv 3 \\
\end{cases} \qquad \mod{6};
 $$
 the linear term varies with period 12, given by 
 $$
 \hat a_1(t) = \begin{cases}
 \frac{31}{15}\,,		&\text{ if } t \equiv 0 \\
 \frac{103}{120}\,,		&\text{ if } t \equiv 1,5 \\
 \frac{133}{120}\,,		&\text{ if } t \equiv 2,10 \\
 \frac{47}{30}\,,		&\text{ if } t \equiv 3 \\
 \frac{37}{30}\,,		&\text{ if } t \equiv 4,8 \\
 \frac{233}{120}\,,		&\text{ if } t \equiv 6 \\
 \frac{11}{15}\,,		&\text{ if } t \equiv 7,11 \\
 \frac{203}{120}\,,		&\text{ if } t \equiv 9 \\
\end{cases} \qquad \mod{12};
 $$
and $\hat a_0$, given by Table \ref{Tb:magilatinaffsymc0}, varies with period 120.

\begin{table}[htb]
\newcommand\Strut{\rule{0ex}{2.5ex}}
\begin{tabular}{|r|c||r|c||r|c||r|c||r|c||r|c|}
\hline
\Strut $t$	& $\hat a_0(t)$	&$t$	& $\hat a_0(t)$	&$t$	& $\hat a_0(t)$	&$t$	& $\hat a_0(t)$	&$t$	& $\hat a_0(t)$	&$t$	& $\hat a_0(t)$	\\[3pt]
\hline\hline
\Strut 0&8			&20&$\frac{47}{18}$	&40&$\frac{31}{9}$	&60&$\frac{15}{2}$	&80&$\frac{28}{9}$	&100&$\frac{53}{18}$	\\[3pt]
\hline
\Strut 1&$\frac{2027}{2880}$	&21&$\frac{1523}{320}$	&41&$\frac{1067}{2880}$	&61&$\frac{3467}{2880}$ &81&$\frac{1363}{320}$	&101&$\frac{2507}{2880}$\\[3pt]
\hline
\Strut 2&$\frac{553}{360}$	&22&$\frac{493}{360}$	&42&$\frac{257}{40}$	&62&$\frac{373}{360}$	&82&$\frac{673}{360}$	&102&$\frac{237}{40}$	\\[3pt]
\hline
\Strut 3&$\frac{979}{320}$	&23&$\frac{-949}{2880}$	&43&$\frac{-1429}{2880}$&63&$\frac{1139}{320}$	&83&$\frac{-2389}{2880}$&103&$\frac{11}{2880}$	\\[3pt]
\hline
\Strut 4&$\frac{229}{90}$	&24&$\frac{38}{5}$	&44&$\frac{199}{90}$	&64&$\frac{137}{45}$	&84&$\frac{71}{10}$	&104&$\frac{122}{45}$	\\[3pt]
\hline
\Strut 5&$\frac{847}{576}$	&25&$\frac{751}{576}$	&45&$\frac{343}{64}$	&65&$\frac{559}{576}$	&85&$\frac{1039}{576}$	&105&$\frac{311}{64}$ 	\\[3pt]
\hline
\Strut 6&$\frac{221}{40}$	&26&$\frac{409}{360}$	&46&$\frac{349}{360}$	&66&$\frac{241}{40}$	&86&$\frac{229}{360}$	&106&$\frac{529}{360}$	\\[3pt]
\hline
\Strut 7&$\frac{1739}{2880}$	&27&$\frac{1171}{320}$	&47&$\frac{779}{2880}$	&67&$\frac{299}{2880}$	&87&$\frac{1331}{320}$	&107&$\frac{-661}{2880}$\\[3pt]
\hline
\Strut 8&$\frac{104}{45}$	&28&$\frac{193}{90}$	&48&$\frac{36}{5}$	&68&$\frac{163}{90}$	&88&$\frac{119}{45}$	&108&$\frac{67}{10}$	\\[3pt]
\hline
\Strut 9&$\frac{1427}{320}$	&29&$\frac{3083}{2880}$	&49&$\frac{2603}{2880}$	&69&$\frac{1587}{320}$	&89&$\frac{1643}{2880}$	&109&$\frac{4043}{2880}$\\[3pt]
\hline
\Strut 10&$\frac{149}{72}$	&30&$\frac{49}{8}$	&50&$\frac{125}{72}$	&70&$\frac{113}{72}$	&90&$\frac{53}{8}$	&110&$\frac{89}{72}$	\\[3pt]
\hline
\Strut 11&$\frac{-1813}{2880}$	&31&$\frac{587}{2880}$	&51&$\frac{1043}{320}$	&71&$\frac{-373}{2880}$	&91&$\frac{-853}{2880}$	&111&$\frac{1203}{320}$	\\[3pt]
\hline
\Strut 12&$\frac{73}{10}$	&32&$\frac{131}{45}$	&52&$\frac{247}{90}$	&72&$\frac{39}{5}$	&92&$\frac{217}{90}$	&112&$\frac{146}{45}$	\\[3pt]
\hline
\Strut 13&$\frac{2891}{2880}$	&33&$\frac{1299}{320}$	&53&$\frac{1931}{2880}$	&73&$\frac{1451}{2880}$	&93&$\frac{1459}{320}$	&113&$\frac{491}{2880}$ \\[3pt]
\hline
\Strut 14&$\frac{301}{360}$	&34&$\frac{601}{360}$	&54&$\frac{229}{40}$	&74&$\frac{481}{360}$	&94&$\frac{421}{360}$	&114&$\frac{249}{40}$	\\[3pt]
\hline
\Strut 15&$\frac{279}{64}$	&35&$\frac{-17}{576}$	&55&$\frac{463}{576}$	&75&$\frac{247}{64}$	&95&$\frac{271}{576}$	&115&$\frac{175}{576}$	\\[3pt]
\hline
\Strut 16&$\frac{128}{45}$	&36&$\frac{69}{10}$	&56&$\frac{113}{45}$	&76&$\frac{211}{90}$	&96&$\frac{37}{5}$	&116&$\frac{181}{90}$	\\[3pt]
\hline
\Strut 17&$\frac{2219}{2880}$	&37&$\frac{4619}{2880}$	&57&$\frac{1491}{320}$	&77&$\frac{3659}{2880}$	&97&$\frac{3179}{2880}$	&117&$\frac{1651}{320}$	\\[3pt]
\hline
\Strut 18&$\frac{233}{40}$	&38&$\frac{157}{360}$	&58&$\frac{457}{360}$	&78&$\frac{213}{40}$	&98&$\frac{337}{360}$	&118&$\frac{277}{360}$	\\[3pt]
\hline
\Strut 19&$\frac{-277}{2880}$	&39&$\frac{1267}{320}$	&59&$\frac{-1237}{2880}$&79&$\frac{1163}{2880}$	&99&$\frac{1107}{320}$	&119&$\frac{203}{2880}$	\\[3pt]
\hline
\end{tabular}
\vspace{15pt}
\caption{Constant terms of the truncated constituents of $\hL_\aa(t)$, the number of symmetry types of magilatin squares with given magic sum.}
\label{Tb:magilatinaffsymc0}
\end{table}

The period of $\hL_\aa$ is 840.  The principal constituent, that for $t\equiv0 \mod{840}$, is 
$$
\frac{1}{576} t^4 - \frac{1}{48} t^3 + \frac{25}{96} t^2 - \frac{74}{35} t + 9 .
$$
(This incorporates the $-S_7$ term.)

%%%%%%
\subsubsection{Some numbers}
\label{mla3numbers}

The first several nonzero values are given in the table.  The third line gives the number of symmetry classes of squares.
\[
\begin{tabular}{r||c|c|c|c|c|c|c|c|c|c|c|c|c|c|c|}
$t$     	&6&7&8&9&10&11&12&13&14&15&16&17&18&19\\
\hline\hline
$L_\aa(t)$	&12&12&24&72&156&240&552&600&1020&1548&2004&2568&4008&4644\\
\hline
$\hL_\aa(t)$    &1&1&2&4&7&10&20&22&35&50&63&78&116&131\\
\hline
$R_{\mla}(t)$    &12& 12&24&60&144&216&480&444&780&996&1404&1548&2640&3696\\
\hline
$r_{\mla}(t)$    &1&1&2&3&6&8&16&15&25&30&41&43&66&68\\
\end{tabular}
\]
The sequences $L_\aa(t)$ and $\hL_\aa(t)$ are A173549 and A173730 in the OEIS \cite{OEIS}.  The numbers $R_\mla(t)$ of reduced magilatin squares and $\hR_\mla(t)$ of normalized, reduced squares with largest value $t$ are sequences A174020 and A174021.

%%%%%%%%%%%%%%%%%%%%%%%%%%%%%%%%%%%%%%%%%%%%%%%%%%%%%%%%%%%%%%%%%
\section{Observations and conjectures} \label{questions}

A remarkable fact is that the period of every one of our strong Ehrhart quasipolynomials equals the denominator, when it could be much smaller.

For some small values of $t$ we calculated $M_\aa(t)=E_\PoH^\circ(t)$ by hand, which is feasible because the problem is 2-dimensional.  The process of counting lattice points in a diagram drew our attention to some remarkable phenomena that apply to the semimagic and magilatin problems as well.  Let $\delta := \dim s$; let $c_k$ be the coefficient in the quasipolynomial $E_\PoH^\circ(t)$ and let $c_k^w$ be that in the Ehrhart quasipolynomial of $P$, and let $p_k$, $p_k^w$ be their periods.  We observe that the variation in $c_{\delta-1}$ is exactly the same as that in $c_{\delta-1}^w$, i.e.,
$$
c_{\delta-1}(t) - c_{\delta-1}(t-1) = c_{\delta-1}^w(t) - c_{\delta-1}^w(t-1) ,
$$
but that is not so for most lower coefficients, especially $c_0$.  The reason is that adding each new excluded hyperplane results in a constant deduction in degree $\delta-1$ (as we discussed at Equation (4.9) in our first article \cite{IOP}) but a more irregular one in lower terms.  We observe that $p_k$ increases---that is, there is longer-term variation in $c_k$---as $k$ decreases in every case.  Thus we propose some daring conjectures.  

\begin{conj} \label{Cj:periods}
In an inside-out counting problem, let $\delta := \dim P$.
 \begin{enumerate}
 \item[(a)] $p_k^w \mid p_k$ for $0\leq k \leq \delta$.  (We know that $p_{\delta-1} = p_{\delta-1}^w$ because $c_{\delta-1}$ and $c_{\delta-1}^w$ have the same variation.)
 \item[(b)] If $p_j = p_j^w$ for all $j \geq k$, then the variation in $c_k$ is the same as that in $c_k^w$.
 \item[(c)] The period ratios increase by a multiplicative factor as $k$ decreases:
 $$
 \frac{p_k}{p_k^w} \ \Big| \ \frac{p_{k-1}}{p_{k-1}^w} \quad \text{for } 0\leq k \leq \delta .
 $$
\end{enumerate}
We do not suggest $p_k \mid p_{k+1}$ because that is false in general in ordinary Ehrhart theory, according to McAllister and Woods \cite{McAllWoods}.  However, it might be true for the kinds of inside-out polytopes that arise in cubical and affine counting.
 \end{conj}

%%%%%%%%%%%%%%%%%%%%%%%%%%%%%%%%%%%%%%%%%%%%%%%%%%%%%%%%%%%%%%%%%

\bigskip\hrule\bigskip

(Concerned with sequences A108576--A108579, A173546--A173549, A173723--A173730, A174018--A174021, A174256--A174257.)

\bigskip\hrule\bigskip

\end{document}